\documentclass[11pt]{amsart}
\usepackage{amsmath,amsthm,amscd,amssymb}
\usepackage{geometry}
\usepackage{pb-diagram}
\usepackage{graphicx}

\usepackage{amssymb}

\usepackage[dvipsnames,usenames]{color}
\usepackage[colorlinks=true, urlcolor=NavyBlue, linkcolor=NavyBlue, citecolor=NavyBlue]{hyperref}

\newtheorem{thm}{Theorem}[section]
\newtheorem{lem}[thm]{Lemma}
\newtheorem{prop}[thm]{Proposition}
\newtheorem{cor}[thm]{Corollary}

\theoremstyle{definition}
\newtheorem{ex}[thm]{Example}
\newtheorem{defn}[thm]{Definition}

\newtheorem{rem}[thm]{Remark}

\numberwithin{equation}{section}
\numberwithin{figure}{section}

\newcommand{\Z}{\mathbb{Z}}
\newcommand{\K}{\mathcal{K}}
\newcommand{\C}{\mathcal{C}}
\newcommand{\cP}{\mathcal{P}}
\newcommand{\cN}{\mathcal{N}}

\newcommand{\T}{\mathcal{T}}
\newcommand{\cB}{\mathcal{B}}

\newcommand{\spin}{{\mathrm{spin}}}
\newcommand{\spinc}{{\mathrm{spin^c}}}
\newcommand{\s}{{\mathfrak{s}}}
\renewcommand{\t}{{\mathfrak{t}}}
\renewcommand{\a}{\alpha}

\newcommand{\lk}{{\mathrm{lk}}}
\newcommand{\im}{{\mathrm{im\ }}}
\newcommand{\del}{\partial}
\newcommand{\inv}{^{-1}}

\newcommand{\onto}{\twoheadrightarrow}

\newcommand{\G}{\Gamma}

\newcommand{\Q}{\mathbb{Q}}
\newcommand{\R}{\mathcal{R}}

\newcommand{\Hom}{\operatorname{Hom}}

\usepackage{pstricks}
\usepackage{pst-node}

\renewcommand{\P}{{\mathcal{P}}}
\newcommand{\sss}{\mathfrak{s}}
\newcommand{\cfkinf}{CFK^\infty}
\newcommand{\x}{\mathbf{x}}

\title[Filtering Topologically Slice Knots]{Filtering Smooth Concordance classes of Topologically Slice Knots}

\author{Tim D. Cochran$^{\dag}$}
\address{Department of Mathematics MS-136, P.O. Box 1892, Rice University, Houston, TX 77251-1892}
\email{cochran@rice.edu}

\author{Shelly Harvey$^{\dag\dag}$}
\address{Department of Mathematics MS-136, P.O. Box 1892, Rice University, Houston, TX 77251-1892}
\email{shelly@rice.edu}

\author{Peter Horn$^{\dag\dag\dag}$}
\address{Department of Mathematics, Syracuse University, 215 Carnegie Building, Syracuse, NY, 13244-1150}
\email{pdhorn@syr.edu}

\thanks{$^{\dag}$Partially supported by the National Science Foundation  DMS-1006908}
\thanks{ $^{\dag\dag}$Partially supported by NSF CAREER DMS-0748458}
\thanks{$^{\dag\dag\dag}$Partially supported by NSF Postdoctoral Fellowship DMS-0902786}

\subjclass[2000]{Primary Secondary}

\begin{document}
\date{\today}
\begin{abstract} We propose and analyze a structure with which to organize the difference between a knot in $S^3$ bounding a topologically embedded $2$-disk in $B^4$ and it bounding a smoothly embedded disk. The $n$-solvable filtration of the \textit{topological} knot concordance group, due to Cochran-Orr-Teichner, may be complete in the sense that any knot in the intersection of its terms may well be topologically slice. However, the natural extension of this filtration to what is called the $n$-solvable filtration of the \textit{smooth} knot concordance group,  is unsatisfactory because \textit{any} topologically slice knot lies in \textit{every} term of the filtration.  To ameliorate this we investigate a new filtration, $\{\cB_n\}$,  that is simultaneously a refinement of the $n$-solvable filtration and a generalization of notions of positivity studied by Gompf and Cochran. We show that each $\cB_n/\cB_{n+1}$  has infinite rank. But our primary interest is in the induced filtration, $\{\T_n\}$, on the subgroup, $\T$, of knots that are topologically slice. We prove that $\T/\T_0$ is large, detected by gauge-theoretic invariants and the $\tau$, $s$, $\epsilon$-invariants; while the non-triviliality of  $\T_0/\T_1$ can be detected by certain $d$-invariants. All of these concordance obstructions vanish for knots in $\T_1$. Nonetheless, going beyond this, our main result is that $\T_1/\T_2$ has positive rank. Moreover under a ``weak homotopy-ribbon'' condition, we show that each  $\T_n/\T_{n+1}$ has positive rank. These results suggest that, even among topologically slice knots, the fundamental group is responsible for a wide range of complexity.
\end{abstract}

\maketitle

\section{Introduction}\label{sec:intro}

\vspace{.2in}

One of the most surprising mathematical developments of the last $30$ years  was the discovery that $\mathbb{R}^4$, in stark contrast to all other dimensions,  has an infinite number of inequivalent differentiable structures.  This was a consequence of the work of Fields medallists Michael Freedman and Simon Donaldson ~\cite{Don1,FQ}. In the intervening years many topological $4$-manifolds have been shown to admit an infinite number of smooth structures, distinct up to diffeomorphism. Indeed, as of this writing there is not a single topological $4$-manifold that is known to admit a \textit{finite} (non-zero) number of smooth structures. This striking difference between the topological and smooth categories can, in a sense, be traced to the failure of the Whitney Trick in dimension $4$ ~\cite[Thm. 9.27]{GompfStip}. This may be thought of as the inability to replace a topologically embedded $2$-dimensional disk by a smoothly embedded disk. Locally, given a $2$-disk, $\Delta$, topologically  embedded in the $4$-ball so that $\partial \Delta$ is a \textit{knot} $K$ in $S^3\equiv \partial B^4$, we cannot necessarily find a smoothly embedded disk with $K$ as boundary (as first investigated by Fox, Milnor and Kervaire in the 1950's).  Thus this local failure may be viewed as a paradigm  for the chasm between the categories on the global scale of $4$-manifolds.

Despite this proliferation of smooth structures on $4$-manifolds the authors know of no attempt to organize the set of all such structures. Here we propose a scheme for organizing this difference in categories for the local problem. A knotted circle $K$ in $S^3$ is said to be \textbf{topologically slice} if it is the boundary of a \textit{topologically} embedded $2$-disk (with a product regular neighborhood) in $B^4$. A knot is said to be a \textbf{slice knot} is the boundary of a \textit{smooth} embedding of a $2$-disk in $B^4$. We propose a method for organizing the difference between these notions (generalizing ~\cite{Go1}). We also give examples exhibiting  new behavior among knots that are topologically slice but not smoothly slice. Our proposed organizational scheme uses a known group structure on certain equivalence classes of knots, which we now review.

A \textbf{knot} $K$ is the image of a tame embedding of an oriented circle into $S^3$. Two knots, $K_0\hookrightarrow S^3\times \{0\}$ and $K_1\hookrightarrow S^3\times \{1\}$, are (smoothly) \textbf{ concordant} if there exists a proper smooth embedding of an annulus into $S^3\times [0,1]$ that restricts to the knots on $S^1\times \{0,1\}$. Let $\mathcal{C}$ denote the set of  concordance classes of knots. It is known that the connected sum operation endows $\mathcal{C}$ with the structure of an abelian group, called the \textbf{smooth knot concordance group}. The identity element is the class of the trivial knot. It is elementary to see that this equivalence class is precisely the set of slice knots.  The inverse of $K$ is the class of the mirror-image of $K$ with the circle orientation reversed, denoted $-K$.

In ~\cite{COT} a  filtration  by subgroups of the \textit{topological knot concordance group} was defined. This filtration has provided a convenient framework for many recent advances in the study of knot concordance. For example, the classical invariants of Milnor, Levine, Tristram, and Casson-Gordon are encapsulated in the low-order terms. The  filtration is also significant because of its strong natural connection to the techniques of A. Casson and M. Freedman on the \emph{topological} classification problem for $4$-manifolds.  This filtration my be complete, in that its intersection may be precisely the set of topologically slice knots.  More recent papers  (e.g. ~\cite[p.1423]{CHL3}\cite{CHL6}) and the present paper are concerned only with a filtration of the \textit{smooth} knot concordance group \textit{suggested} by that in ~\cite{COT}),
$$
\cdots \subset \mathcal{F}_{n+1} \subset \mathcal{F}_{n.5}\subset\mathcal{F}_{n}\subset\cdots \subset
\mathcal{F}_1\subset \mathcal{F}_{0.5} \subset \mathcal{F}_{0} \subset \mathcal{C},
$$
called the \textbf{($n$)-solvable filtration} of $\mathcal{C}$.  The $n$-solvable filtration (just like the filtration of the topological concordance group) is highly non-trivial; each of the associated graded abelian groups $\{\mathcal{F}_{n}/\mathcal{F}_{n.5}~|~n\in\mathbb{N}\}$ contains $\Z^\infty\oplus \Z^\infty_2$ ~\cite{CHL3,CHL6}. 

However the $n$-solvable filtration of the smooth knot concordance group is not useful in distinguishing among knots that are slice in the \emph{topological category}, but not slice in the smooth category, which is the main focus of this paper. Indeed, if $\mathcal{T}$ denotes the subgroup of smooth concordance classes of knots that are topologically slice, then it was observed in ~\cite[p.1423]{CHL3} (using ~\cite[Section 8.6]{FQ}) that
$$
\mathcal{T}\subset\cap_{n=1}^\infty\mathcal{F}_n.
$$
Yet $\mathcal{T}$ itself is known to be highly non-trivial. It was first shown in ~\cite{Endo} using gauge-theoretic techniques of Furuta and Fintushel-Stern that $\mathcal{T}$ has infinite rank. Recent work of Hedden-Livingston-Ruberman, Hom and Hedden-Kirk shows that much finer structure exists in $\T$ ~\cite{HLR,Hom,HedKirk2}. Yet no proposal has been made to organize this structure.

It is the purpose of the present work to propose and investigate new  filtrations of $\C$ that, like the $n$-solvable filtration, are highly non-trivial and yet are superior to that filtration in that they induce non-trivial filtrations of $\mathcal{T}$. Our filtration can also been seen as a generalization of Gompf's notion of \textit{kinkiness} ~\cite{Go1}. Our filtration thus retains the strong connection to the tower techniques of A. Casson and M. Freedman.

In Section~\ref{sec:filtrations}, we define nested submonoids, $\{\mathcal{P}_n\}$ and  $\{\mathcal{N}_n\}$,  of $\C$, which we call the  $n$-positive and $n$ negative knots, respectively. A knot is \textbf{$0$-positive (respectively $0$-negative)} if it bounds a smoothly embedded $2$-disk, $\Delta$, in a smooth, compact, oriented, simply-connected $4$-manifold, $V$,  with $\partial V=S^3$, where the intersection form on $H_2(V)$ is positive definite (respectively negative definite), and where $[\Delta]=0$ in $H_2(V,\partial V)$. It follows that $V$ is homeomorphic to a punctured $\#_1^m\mathbb{C}P(2)$. To motivate the definition of $\cP_n$ for $n>0$, note that if $V$ were diffeomorphic to a punctured $\#_1^m\mathbb{C}P(2)$ and if the corresponding $\mathbb{C}P(1)$'s were embedded in the complement of $\Delta$, they could be (smoothly)  blown-down, resulting in  an actual slice disk in $B^4$.  Therefore we say a knot $K$ is \textbf{$n-$positive}  if $H_2(V)$ admits a basis of disjointly embedded surfaces \textit{in the exterior of $\Delta$} , where, loosely-speaking,  these surfaces are more and more like $2$-spheres as $n$ increases (the precise definitions are in Section~\ref{sec:filtrations}). The sets of $n$-positive and $n$-negative knots induce monoid  filtrations of $\C$, called \textbf{the $n$-positive filtration}, $\{\mathcal{P}_n\}$, and
the \textbf{$n$-negative filtration}, $\{\mathcal{N}_n\}$ respectively.  Then we observe that the intersection $\cB_n\equiv \mathcal{N}_n\cap\mathcal{P}_n$, which we call \textbf{n-bipolar knots}, yields a filtration of $\C$ by \textit{subgroups}:
$$
\{0\}\subset\dots\subset\mathcal{B}_{n+1}\subset\mathcal{B}_n\subset\dots \subset\mathcal{B}_0\subset \C.
$$

We show that membership in $\mathcal{P}_0$ (and $\mathcal{N}_0$) is obstructed by (the sign of) many well-known knot concordance invariants.  For example,

\begin{prop}\label{prop:introresults2} Suppose $K$ is a knot.
\begin{enumerate}
\item [1.] If $K\in \cB_0$ then the Levine-Tristram signature function of $K$ vanishes, so $K$ has finite order in the algebraic concordance group (Corollary~\ref{cor:Zofiniteorder}) ;
\item [2.] If $K\in \cB_1$ then $K$ is algebraically slice (Corollary~\ref{cor:ponealgslice}) ;
\item [3.] If $K\in \cB_2$ then its Casson-Gordon slicing obstructions vanish, as do all metabelian signature obstructions (Theorems~\ref{thm:CGvanish}and ~\ref{thm:generalsignaturesobstruct}) .
\end{enumerate}
\end{prop}

In fact, in Section~\ref{sec:relationsn-solv} we show that these new filtrations are (essentially) refinements of $\{\mathcal{F}_n\}$ and so the $n^{th}$-order signature obstructions that were used to study $\mathcal{F}_n$ also obstruct membership in $\cB_n$. Thus, even though the terms of $\{\cB_n\}$ are much smaller than those of $\{\mathcal{F}_n\}$, we are still able to show that the new filtration is highly non-trivial using the same techniques as were used to show that $\{\mathcal{F}_n\}$ was non-trivial.

\newtheorem*{thm:nontrivialityoffiltration}{Theorem~\ref{thm:nontrivialityoffiltration}}
\begin{thm:nontrivialityoffiltration} For each $n\geq 1$, there exists
$$
\Z^\infty\subset \frac{\cB_n}{\cB_{n+1}}.
$$
while for each $n\geq 0, n\neq 1$ there exists
$$
\Z_2^\infty\subset \frac{\cB_n}{\cB_{n+1}}.
$$
\end{thm:nontrivialityoffiltration}

But the examples of Theorem~\ref{thm:nontrivialityoffiltration} are not topologically slice. So we turn our attention to the intersection $\mathcal{T}_n\equiv \cB_n\cap\mathcal{T}$ which yields a filtration of $\mathcal{T}$ by subgroups:
$$
\{0\}\subset\dots\subset\mathcal{T}_{n+1}\subset\mathcal{T}_n\subset\dots \subset\mathcal{T}_0\subset \mathcal{T}.
$$

The advantage of  $\{\cB_n\}$ over $\{\mathcal{F}_n\}$ is that $\{\T_n\}$ is an interesting non-trivial filtration of $\T$ whereas $\{\mathcal{F}_n\cap \T\}$ is a trivial filtration of $\T$ (each term is $\T$ itself).  Evidence that $\{\T_n\}$ is natural is provided by showing that known invariants fit well into this structure. We are able to analyze most of the known invariants that obstruct a knots being smoothly slice and prove (more generally) that they  obstruct membership in certain terms of $\{\mathcal{T}_n\}$. Evidence for non-triviality is provided by showing non-triviality for certain successive quotients $\T_{n}/\T_{n+1}$.

Specifically, we show that, even among topologically slice knots, membership in $\mathcal{P}_0$ is obstructed by (the sign of) many well-known knot concordance invariants. 

\begin{prop}\label{prop:introresults1} If $K\in \cP_0$ then
\begin{itemize}
\item  [1.] the Levine-Tristram signature function of $K$ is non-positive (Proposition~\ref{prop:easy2fold}) ;
\item  [2.] $\tau(K)\geq 0$ (Ozsv\'{a}th-Szab\'{o}  see Proposition~\ref{prop:tau});
\item  [3.] $s(K)\geq 0$ (Kronheimer-Mrowka see Proposition~\ref{prop:s}  ) ;
\item  [4.] If, additionally, the $p^r$-signatures of $K$ vanish and the $p^r$-fold  cover of $S^3$ branched over $K$ is a homology sphere, then $\delta_p(K)\leq 0$ (Corollary~\ref{cor:homologyspheredinvariants}) ;
\item [5.] if $\Sigma$ is $\pm 1$-surgery on $K$ then $d(\Sigma)\leq 0$ (Corollary~\ref{cor:dinvariantssurgeries}) ;
\item [6.]  If, additionally, the $2^r$-signatures of $K$ vanish, then the corresponding Hedden-Kirk slice obstructions (extending the Fintushel-Stern invariants) obstruct membership in  $\cP_0$ (see Theorem~\ref{thm:Endo});
\item [7.] If $K\in \cB_0$, then  $\epsilon(K)=0$ (see Proposition~\ref{prop:epsilon}).
\end{itemize}
Here $\tau$ is the concordance invariant of Ozsv\'{a}th-Szab\'{o} and Rasmussen defined from Heegard Floer homology ~\cite{OzSz2}, $\epsilon$ is the concordance invariant of Hom ~\cite{Hom}, $s$ is Rasmussen's concordance invariant defined from Khovanov homology ~\cite{Ras1}, and $\delta_p$ refers to the invariants of Manolescu-Owens ~\cite{ManOw} and Jabuka ~\cite{Jab} (Ozsv\'{a}th-Szab\'{o} d-invariants associated to prime-power branched covers and certain specific $\spinc$-structures ~\cite{OzSz1}). The term Fintushel-Stern obstructions (as generalized in ~\cite[Theorem 1]{HedKirk2}) refers to the  invariant of those authors that obstructs a rational homology $3$-sphere from being the boundary of a $4$-manifold with positive definite intersection form ~\cite{FiSt2}, as applied to a $2^r$-fold cyclic cover of $S^3$ branched over $K$.
\end{prop}
If $K\in \mathcal{N}_0$ then a similar result holds, so that if $K\in \cB_0$ then the invariants in Proposition~\ref{prop:introresults1} $1.-5.$ and $7.$ are zero.

In fact we show that:

\newtheorem*{thm:Endo}{Theorem~\ref{thm:Endo}}
\begin{thm:Endo}
The family of topologically slice pretzel knots considered by Endo generates a
$$
\Z^\infty\subset \T/\T_0.
$$
The proof uses Endo's original argument. Similarly, using the extension of the Fintushel-Stern-Furuta strategy (and the calculations) due to Hedden-Kirk ~\cite[Theorem 1]{HedKirk2}, the set of Whitehead doubles of the $(2,2^n-1)$ torus knots ($n>1$) has the same property.
\end{thm:Endo}

Hence  $\T/\T_0$ is quite rich and the invariants of Proposition~\ref{prop:introresults1} are very useful for proving that a knot is not in $\mathcal{T}_0$, but none is directly useful beyond that.

In Section~\ref{sec:firstlevel} we  prove that the (signs of) Ozsv\'ath-Szabo $d$-invariants associated to prime-power branched covers can do slightly better: they obstruct membership in $\cP_1$ hence obstruct membership in $\mathcal{T}_1$.

\newtheorem*{thm:onenegativedinvariants}{Theorem~\ref{thm:onenegativedinvariants}}
\begin{thm:onenegativedinvariants} If $K\in\cP_1$  and $Y$ is the $p^r$-fold cyclic  cover of $S^3$ branched over $K$, then there is a metabolizer $G<H_1(Y)$ for the $\Q/\Z$-linking form on $H_1(Y)$; and there is a $\spinc$ structure $\mathfrak{s}_0$ on $Y$  such that $d(Y,\mathfrak{s}_0+\hat{z})\leq 0$ for all $z\in G$, where $\hat{z}$ is the Poincare dual of $z$.  Furthermore we may take $\mathfrak{s}_0$ to be a $\spinc$ structure corresponding to a spin structure on $Y$.  
\end{thm:onenegativedinvariants}

A similar but sharper result (Corollary~\ref{cor:zerodeltainvariants}) holds for the $\delta_p$-invariants.  Taken together (varying $p$) these invariants yield a homomorphism 
$$
\frac{\mathcal{T}_0}{\mathcal{T}_{1}}\to \times_{i=1}^{\infty} \Z,
$$
but as of now too few calculations have been done for topologically slice knots to prove that the image is infinitely generated.

None of the invariants above is capable of detecting non-triviality in $\mathcal{T}_1/\mathcal{T}_{2}$. Specifically, the Casson-Gordon slicing obstructions and the $d$-invariant slicing obstructions vanish for knots in $\mathcal{T}_1$. Given this, it is surprising that we can show:

\newtheorem*{thm:maintheoremtop}{Theorem~\ref{thm:maintheoremtop}}
\begin{thm:maintheoremtop}  The group
$$
\frac{\mathcal{T}_1}{\mathcal{T}_{2}}
$$
has positive rank.
\end{thm:maintheoremtop}

This is shown using a \textit{combination} of $d$-invariants and Casson-Gordon invariants.

We also sketch, in Theorem~\ref{thm:generalnontriviality}, the proof of a result only slightly weaker than the desired end result that  $\mathcal{T}_n/\mathcal{T}_{n+1}$ is non-zero for every $n$. Namely we  exhibit topologically slice knots in $\mathcal{T}_n$ that do not lie $\mathcal{T}_{n+1}$ in a ``weakly homotopy ribbon'' fashion (there exists no $(n+1)$-positon $V$ as in Definition~\ref{def:positive} wherein the inclusion $S^3-K\to V-\Delta$ induces a surjection on Alexander modules). This is shown using a combination of $d$-invariants and von Neumann signature invariants.

There is still a lot of room for improvement.  Neither of $Wh(RHT)$ and $Wh(Wh((RHT))$ (iterated Whitehead doubles of the right-handed trefoil) lie in $\mathcal{T}_0$ (as detected, say, by their $\tau$-invariants), yet we offer no new invariants with which to distinguish them. Moreover, the positive and bipolar filtrations are still not as discriminating as we could hope for certain knots with Alexander polynomial $1$. For example, in Corollary~\ref{prop:Alexone} we show that the (untwisted) Whitehead double of any knot in  $\mathcal{B}_0$ in fact lies in  $\mathcal{T}_n$ for \textit{every} $n$. Specifically, the Whitehead double of the figure eight knot lies in the intersection of all $\mathcal{T}_n$. To detect such knots, a different filtration and truly new invariants are needed.

Finally we remark that, although the bipolar filtration is far superior to the solvable filtration when studying $\mathcal{T}$, when considering the entire concordance group $\mathcal{C}$, the solvable filtration is still useful (perhaps more useful). Indeed $\cB_m\cap\mathcal{F}_n$ is a bifiltration that is finer that either individual filtration.

\section{Definitions of the Filtrations}\label{sec:filtrations}

In this section we define various new relations on $\mathcal{C}$ and use these relations to define the new filtrations of $\C$ (and $\T$) that will be our objects of study. To accomplish this one should consider relaxing the condition that a knot bounds an embedded $2$-disk in $B^4$. There are two obvious paths (although they are not unrelated). One possibility is to relax the condition on the $2$-disk and ask only that the knot bound a singular disk or a surface or a grope, for example.  Alternatively, one  can relax the condition on $B^4$, and consider when a knot  bounds an embedded disk in some other $4$-manifold. Here we take the latter approach.

 We say that  two knots $K$ and $K'$ are \textbf{concordant in  $V$} if $V$ is a smooth, compact, oriented, $4$-manifold with $\partial V\cong S^3 \coprod -S^3$ and there exists an annulus, $A$, smoothly and properly embedded in $V$ whose boundary gives the knots $K$ and $-K'$ and where the annulus is trivial in $H_2(V,\partial V)$.  We say that $K$ \textbf{is slice in $V$} if $V$ is smooth, compact and oriented with $\partial V\cong S^3$ and there is a $2$-disk smoothly embedded in $V$ whose boundary is $K$, and where the slice disk is trivial in $H_2(V,S^3)$. This last condition is important because any knot bounds a smoothly embedded disk in a punctured connected sum of $\mathbb{C}P(2)$'s ~\cite[Lemma 3.4]{CLick}.

 In ~\cite[Def. 2.1]{CGompf}  Cochran and Gompf defined a relation, $K\geq K'$, on knots that generalized the relation that $K$ can be transformed to $K'$ by changing only positive crossings. We generalize and filter this notion by adding an integer parameter.

\begin{defn}\label{def:geq} We say $K\geq_n K'$ if $K$ is concordant to $K'$ in a smooth $4$-manifold $V$ such that 
\begin{itemize}
\item [1.] $\pi_1(V)=0$;
\item [2.] the intersection form on $H_2(V)$ is positive definite;
\item [3.] $H_2(V)$ has a basis represented by a collection of surfaces $\{S_i\}$ disjointly and smoothly embedded in the exterior of the annulus $A$ such that, for each $i$, $\pi_1(S_i)\subset \pi_1(V-A)^{(n)}$.
\end{itemize}
\end{defn}

Here $G^{(n)}$ denotes the $n^{th}$ term of the derived series of the group $G$, where $G^{(1)}\equiv [G,G]$ and $G^{(n+1)}\equiv [G^{(n)},G^{(n)}]$. Note that \textit{any} group series satisfying a certain functorialty could be used here and would give, \textit{a priori},  a different relation.

It is elementary to verify  that these relations descend to relations on $\mathcal{C}$. In particular if $K$ is any slice knot, and $U$ is the unknot then $K\geq_n U$ and $U\geq_n K$ for every $n$.  Moreover it is easy to check that each relation is compatible with the monoidal structure on $\C$, that is, if $K\geq_n K'$ then $K\# K''\geq_n K' \# K''$ for any $K''$. These verifications are left to the reader. It follows from $2.$ and $3.$ that the matrix of the intersection form on $H_2(V)$ with respect to this basis is an identity matrix. Condition $3.$ ensures that $A$ is homologically trivial. Each of the relations  $\geq_n$ is clearly reflexive and transitive, but fails to be symmetric and fails to be antisymmetric.  

\begin{defn}\label{def:positive} Let $\mathcal{P}_{n}$, the set of  \textbf{n-positive knots}, be the set of (concordance classes of) knots $K$ such that  $K\geq_n U$ where $U$ is the trivial knot.  Let $\mathcal{N}_{n}$, the set of  \textbf{n-negative knots}, be the set of (concordance classes of) knots $K$ such that  $U\geq_n K$. Equivalently $K$ is \textbf{n-positive} (respectively,  \textbf{n-negative}), if 
$K$ is slice in a smooth $4$-manifold $V$ such that
\begin{itemize}
\item [1.] $\pi_1(V)=0$;
\item [2.] the intersection form on $H_2(V)$ is positive definite (respectively, negative definite);
\item [3.] $H_2(V)$ has a basis represented by a collection of surfaces $\{S_i\}$ disjointly embedded in the exterior  of the slice disk $\Delta$ such that $\pi_1(S_i)\subset \pi_1(V-\Delta)^{(n)}$ for each $i$.
\end{itemize}
Such a $V$ is called an \textbf{$n$-positon for $K$} (respectively an \textbf{$n$-negaton for $K$}).
\end{defn}

Since the surfaces $S_i$ are disjoint, the intersection matrix with respect to this basis is diagonal. Since the intersection form is positive definite and unimodular, this matrix is the identity matrix. Thus these conditions imply (by work of  Freedman on the classification of closed, smooth, simply-connected $4$-manifolds up to homeomorphism) that any  $n$-positon $V$ is a smooth manifold that is homeomorphic to a (punctured) connected sum of $\mathbb{C}P(2)$'s. Hence $V$ is a punctured connected sum of $\mathbb{C}P(2)$'s, but with a possibly exotic differentiable structure. 
 
The parameter $n$ in this definition can be motivated by the following observation:  if the surfaces $S_i$ were actually \emph{spheres} then, since $S_i\cdot S_i=1$ (respectively $-1$), their regular neighborhoods would be diffeomorphic to that of $\mathbb{C}P(1)\subset \pm\mathbb{C}P(2)$ so they could be (differentiably) excised (blown-down), proving that $K$ is smoothly slice in $B^4$ endowed with a possibly exotic smooth structure. Condition $3$.  is intended to progessively approximate, as $n$ increases, this scenario.

\begin{rem} This definition extends to links and string links for which the components have zero pairwise linking numbers.  
\end{rem}

The sets $\mathcal{P}_n$ and $\mathcal{N}_n$ are clearly closed under connected sum. However, if $K\in \mathcal{P}_{n}$ then $-K$ need not  lie in $\mathcal{P}_{n}$ but will certainly lie in $\mathcal{N}_n$. Thus neither $\mathcal{P}_{n}$ nor $\mathcal{N}_{n}$ is a subgroup of $\C$.  However each of the sets in the following definition is a subgroup.

\begin{defn}\label{def:filtr} The set of \textbf{n-bipolar knots} is $\cB_n\equiv \mathcal{N}_n\cap \mathcal{P}_n$.
\end{defn}

Then, in summary, we have:

\begin{prop}\label{prop:filtrationsofC} $\{\mathcal{P}_n\}$ and $\{\mathcal{N}_n\}$ induce filtrations of $\C$ by submonoids
\begin{align*}
\{0\}\subset\dots\subset\mathcal{P}_n\subset\dots \subset\mathcal{P}_0\subset \C,\\
\{0\}\subset\dots\subset\mathcal{N}_n\subset\dots \subset\mathcal{N}_0\subset \C,
\end{align*}
We call these the ($n$)-\textbf{positive filtration} and  ($n$)-\textbf{negative filtration} of $\mathcal{C}$. Moreover both $\{\cB_n\}$ and $\{\langle \mathcal{P}_n\rangle\}$ induce filtrations of $\C$  by \textbf{subgroups}
\begin{align*}
\{0\}\subset\dots\subset\mathcal{B}_n\subset\dots \subset\mathcal{B}_0\subset \C,\\
\{0\}\subset\dots\subset\langle\mathcal{P}_n\rangle\subset\dots \subset\langle\mathcal{P}_0\rangle\subset \C.
\end{align*}
where $\langle \mathcal{P}_n\rangle$ denotes the \textbf{subgroup} of $\C$ generated by the set $\mathcal{P}_n$. The former we call the ($n$)-\textbf{bipolar filtration}.
\end{prop}

Finally, 

\begin{defn}\label{def:topfiltr} Let $\mathcal{T}\subset \C$ denote the subgroup represented by knots that are topologically slice,  and let $\mathcal{T}_n$ denote $\cB_n\cap \mathcal{T}$
\end{defn}

\begin{prop}\label{prop:filtrationsofT} $\{\cP_n\cap\mathcal{T}\}$ and $\{\cN_n\cap\mathcal{T}\}$ are filtrations of $\mathcal{T}$ by submonoids while $\{\mathcal{T}_n\}$  is a filtration of $\mathcal{T}$ by subgroups
\begin{align*}
\{0\}\subset\dots\subset\mathcal{T}_n\subset\dots \subset\mathcal{T}_0\subset \mathcal{T}.
\end{align*}
\end{prop}

We close with several curiosities, the first of which is quite useful later in the paper. 

\begin{cor}\label{cor:torsionfreeC} If $K\in \cP_n$ and $K\notin \cN_n$ then no non-zero multiple of $K$ lies in $\cB_n$. 
\end{cor}
\begin{proof}  For sake of contradiction, suppose $mK\in \cB_n$ for a non-zero integer $m$. Since $\cB_n$ is a group we can assume that $m>0$. Then in particular $mK \in \cN_n$. Since $K\notin \cN_n$,  it follows that $m>1$.   Since $K\in \cP_n$, $-K\in \cN_n$ and so $(m-1)(-K)\in \cN_n$. Since the latter is closed under connected-sum,
$$
K= mK + (m-1)(-K) \in \cN_n,
$$
which is a contradiction.  
\end{proof}

Secondly, if we restrict to Tor($\mathcal{C}$), by which we mean the torsion subgroup of $\mathcal{C}$, the $n$-positive filtration is a filtration by subgroups.

\begin{cor}\label{prop:filtrationsoftorsionC} $\{\mathcal{P}_n\cap \mathrm{Tor}(\mathcal{C})\}$ and $\{\mathcal{N}_n\cap \mathrm{Tor}(\mathcal{C})\}$ are filtrations of $\mathrm{Tor}(\mathcal{C})$  by subgroups, and in fact each equals $\{\cB_n\cap \mathrm{Tor}(\mathcal{C})\}$.
\end{cor}
\begin{proof}[Proof of Corollary~\ref{prop:filtrationsoftorsionC}] Suppose $rK=0$ in $\mathcal{C}$ for some $r>1$.  If $K\in \mathcal{P}_n$ then $-K\in\mathcal{N}_n$ so $(r-1)(-K)\in\mathcal{N}_n$. But since $r(-K)=0$,  $K=(r-1)(-K)$, so $K\in \mathcal{N}_n$. Thus $K\in \cB_n$. This shows that
$$
\mathcal{P}_n\cap \mathrm{Tor}(\mathcal{C})= \cB_n\cap \mathrm{Tor}(\mathcal{C})=\mathcal{N}_n\cap \mathrm{Tor}(\mathcal{C}),
$$
and so each is a subgroup since $\cB_n\cap \mathrm{Tor}(\mathcal{C})$ is the intersection of two subgroups.
\end{proof}

\section{Examples of Knots in $\mathcal{P}_n$, $\mathcal{N}_n$ and $\cB_n$}\label{sec:examples}

In this section we give examples of knots lying deep in the various filtrations.

If $K$ can be transformed to $K'$ by changing some set of positive crossings to negative crossings then $K\geq_0 K'$, as can be seen by blowing-up at the singular points in the trace of the homotopies that accomplish the crossing changes  ~\cite[Lemma 3.4]{CLick}\cite[Prop. 2.2]{CGompf}. Thus it follows that:

\begin{prop}[Cochran-Lickorish]\label{prop:easyexs}  Any knot that can be changed to a slice knot by changing positive crossings (between the same component) to negative crossings lies in $\mathcal{P}_0$. 
\end{prop}
\begin{proof} In ~\cite[Lemma 3.4]{CLick} it is shown that if $K$ satisfies the hypothesis then $K$ is slice in a punctured connected-sum of copies of $\mathbb{C}P(2)$. Thus $K\in \mathcal{P}_0$.
\end{proof}

\begin{ex}\label{ex:easy} Any knot that admits a positive projection lies in $\mathcal{P}_0$. The (twisted or untwisted) Whitehead double of \textit{any knot} using a positive clasp lies in $\mathcal{P}_0$, since it can be unknotted by changing a single positive crossing. The figure $8$ knot lies in $\cB_0$ since it can be unknotted via a positive or a negative crossing.
\end{ex}

\textbf{Question}: Does every strongly quasipositive knot lie in $\mathcal{P}_0$?

\vspace{.1in}

It is easy to create knots and links lying in $\mathcal{P}_n$ or $\cB_n$ using the satellite construction and generalizations of this. In particular suppose that $ST$ is a solid torus embedded in $S^3$ in an unknotted fashion, and  $\eta$ is the oriented meridian circle of $ST$. Suppose $\overline{R}$ is a knot in $ST$ which when  viewed as a knot in $S^3$ will be denoted $R$. Suppose that $\eta\in \pi_1(S^3\setminus R)^{(k)}$. If $k\geq1$ then $\overline{R}$ is said to have \textbf{winding number zero} in $ST$.   Suppose that $J$ is any knot. Then let $R(\eta,J)\equiv \overline{R}(J)$ denote the \textbf{satellite knot} with $\overline{R}$  as pattern and $J$ as companion. This is also called the \textbf{result of infection on $R$ by $J$ along $\eta$}.  It is known that $R(\eta,-)$ induces a well-defined operator on $\C$.

\begin{prop}\label{prop:operatorsact}  With notation in the preceding paragraph, suppose that $R\in \mathcal{P}_n$ (respectively $\mathcal{N}_n$, $\cB_n$),  and $\eta\in \pi_1(S^3\setminus R)^{(k)}$. Then
\begin{align}\label{eq:operators}
R(\eta,\mathcal{P}_{n-k})\subset \mathcal{P}_n;
\end{align}
respectively
\begin{align}\label{eq:operators2}
R(\eta,\mathcal{N}_{n-k})\subset \mathcal{N}_n;
\end{align}
\begin{align}\label{eq:operators3}
R(\eta,\cB_{n-k})\subset \cB_n.
\end{align}
\end{prop}

This generalizes ~\cite[Proposition 2.7]{CGompf}. 

\begin{proof} In the next section we will give an equivalent definition of $\cP_n$ and $\cN_n$. Using the equivalent definition, the proof of Proposition~\ref{prop:operatorsact} is almost identical to that of ~\cite[Lemma 6.4]{CHL4}  (which was done for the $n$-solvable filtration). We include a different proof. 

By symmetry it suffices to prove equation~\ref{eq:operators}. Suppose $J\in \cP_{n-k}$ and let $K\equiv R(\eta,J)$. We will show that $K\in \cP_n$.  Suppose $R$ has slice disk $\Delta_R$ in the $n$-positon $V_R$.  Suppose $J$ has slice disk $\Delta_J$ in the $(n-k)$-positon $V_J$.  We will describe a slice disk for $K$ in an $n$-positon $V_K$.  Recall that $\overline{R}$ lies in an unknotted solid torus $ST$ whose exterior in $S^3$ we denote $ST'$.  The circle $\eta$ may be viewed as a meridian of $ST$ or as a longitude of $ST'$. Form a new $4$-manifold, $V_K$, as the union of $V_R$ and $V_J-(\Delta_J\times \text{int}D^2)$, identifying $ST'$ in the boundary of the former with $\Delta_J\times S^1$ in the boundary of the latter, in such a way that the meridian of $J$ is identified with  $\eta$. One first observes that $\partial V_K$ is the union of $ST$ with $S^3\setminus J$ where the meridian of $J$ is identified with the meridian of $ST$.  Thus $\partial V_K$ is homeomorphic to $S^3$ and the image, under this identification, of the knot $\overline{R}$ is the satellite knot $K=R(\eta, J)$.  Thus  $K$ is slice in $V_K$ (merely by letting $\Delta_K$ be the image of the slice disk $\Delta_R$). 

Since $V_J$ is simply-connected,  $\pi_1(V_J\setminus \Delta_J)$ is normally generated by a meridian of $J$, which has a representative in $\Delta_J\times S^1$. Then, since $V_R$ is simply-connected, it follows from the Seifert-Van Kampen theorem that $V_K$ is simply-connected.

A Mayer-Vietoris sequence shows that
$$
H_2(V_K)\cong H_2(V_R)\oplus H_2(V_J\setminus \Delta_J)\cong H_2(V_R))\oplus H_2(V_J).
$$
By Definition~\ref{def:positive} the latter two groups have bases, $\{\Sigma_i\}$ and $\{S_j\}$, disjoint from $\Delta_R$ and $\Delta_J$ respectively. Thus the union of these bases is a basis for $H_2(V_K)$ consisting of embedded surfaces disjoint from $\Delta_K$. The intersection form with respect to these bases is an identity matrix.

 It remains only to show that these surfaces satisfy the $\pi_1$-condition of Definition~\ref{def:positive}. This is clear for the  $\Sigma_j$. For the surfaces $S_i$ it suffices to show  that:
$$
i_*(\pi_1(V_J\setminus \Delta_J)^{(n-k)})\subset \pi_1(V_K\setminus \Delta_K)^{(n)}.
$$
This follows from two facts. First recall that  $\pi_1(V_J\setminus \Delta_J)$ is normally generated by a meridian, $\mu_J$, of $J$, and that this meridian is identified with $\eta$. Secondly, by hypothesis $\eta\in \pi_1(S^3\setminus R)^{(k)}$,  so $\eta\in \pi_1(V_R\setminus \Delta_R)^{(k)}$.  Hence $i_*(\mu_J)=i_*(\eta)\in \pi_1(V_K\setminus \Delta_K)^{(k)}$. 
\end{proof}

\begin{cor}\label{cor:satellites} Each of the submonoids discussed is closed under forming satellites in the sense that any satellite knot $K=R(\eta,J)$ whose pattern knot $R$ and companion knot $J$ both lie in $\mathcal{P}_n$ (respectively $\mathcal{N}_n$, $\cB_n$) itself lies in $\mathcal{P}_n$ (respectively $\mathcal{N}_n$, $\cB_n$). If the winding number is zero then $K\in \mathcal{P}_{n+1}$ (respectively $\mathcal{N}_{n+1}$, $\cB_{n+1}$).
\end{cor}

Since taking the untwisted Whitehead double of a knot is a satellite operator with winding number zero and with unknotted pattern, we have the following.

\begin{ex}\label{ex:exswhitehead} Let $Wh(-)$ denote the untwisted positive Whitehead double operator (whose clasp has positive crossings). Let $Wh^-(-)$ denote the untwisted negative Whitehead double operator. Now if $J\in \cP_n$, (respectively $\mathcal{N}_n$, $\cB_n$)  then both $Wh(J)$ and $Wh^-(J)$ lie in $\mathcal{P}_{n+1}$ (respectively $\mathcal{N}_{n+1}$, $\cB_{n+1}$). Since the figure eight knot, $E$, lies in $\cB_0$, both $Wh(E)$ and $Wh^-(E)$ lie in $\cB_{1}$.  Since the right-handed trefoil knot, $RHT$, is a positive knot, $Wh^+(RHT)\in \cP_1$ (but not in $\cN_0$ as we shall see later). On the other hand $Wh^{-}(RHT)$ lies in $\cP_1$ and also lies in $\cB_0$ since it can be unknotted by changing the negative crossing undoing the clasp.
\end{ex}

\begin{ex}\label{ex:nine46} Suppose $R$ is the ribbon knot $9_{46}$. 
Let $RHT$ be the right-handed trefoil knot, $LHT$ be the left-handed trefoil knot and  $U$ be the unknot. Then $RHT\in \mathcal{P}_0$ by Proposition~\ref{prop:easyexs}. Note that, $R(LHT,U)$, the knot on the left-hand side of Figure~\ref{fig:ribbonfamily} is also a ribbon knot, as is the knot $R(U,RHT)$ (not pictured). Thus the knot $K=R(LHT,RHT)$ on the right-hand side of Figure~\ref{fig:ribbonfamily} may be viewed as a winding number zero satellite knot with ribbon pattern knot in two different ways and hence we can apply Proposition~\ref{prop:operatorsact} in two different ways to conclude that  
$$
K\equiv R(LHT,RHT)=\left(R(LHT,-)\right)(RHT)\in \mathcal{P}_1,
$$
and
$$
K\equiv R(LHT,RHT)=\left(R(-,RHT)\right)(LHT)\in \mathcal{N}_1.
$$
Thus $K\in \cB_1$.
\begin{figure}[htbp]
\setlength{\unitlength}{1pt}
\begin{picture}(327,151)
\put(0,0){\includegraphics{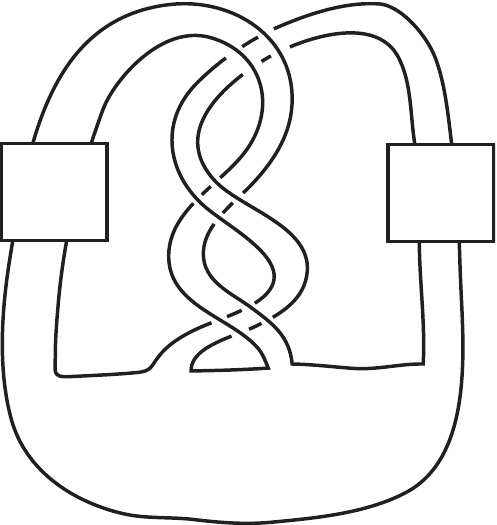}}
\put(185,0){\includegraphics{family_scaled}}
\put(3,93){$LHT$}
\put(123,93){$U$}
\put(300,93){$RHT$}
\put(188,93){$LHT$}
\end{picture}
\caption{$K=R(LHT,RHT)$}\label{fig:ribbonfamily}
\end{figure}

\end{ex}

Unfortunately the positive filtration is still not as discriminating as we would hope for certain kinds of topologically slice knots, namely those with Alexander polynomial $1$.

\begin{cor}\label{prop:Alexone} Suppose  $K=R(\eta,J)$ is a winding number zero satellite knot  whose pattern knot $R$ is a slice knot with Alexander polynomial one, and whose companion knot $J$ lies in $\cP_0$. Then $K\in\cP_n$ for all $n$. Thus the Whitehead double (with either clasp) of a $0$-positive knot is $n$-positive for all $n$.
\end{cor}
\begin{proof}  The winding number zero hypothesis means that $\eta\in \pi_1(S^3\setminus R)^{(1)}$.  Since $R$ has Alexander polynomial one, its Alexander module is trivial, that is, $\pi_1(S^3\setminus R)^{(1)}=\pi_1(S^3\setminus R)^{(2)}$. It follows that $\pi_1(S^3\setminus R)^{(1)}=\pi_1(S^3\setminus R)^{(n)}$ for all $n$. Hence $\eta\in  \pi_1(S^3\setminus R)^{(n)}$ for each $n$. Since $R$ is slice $R\in \cP_n$. Now  Proposition~\ref{prop:operatorsact} with $k=n$ implies the desired result.

\end{proof}

\section{Obstructions to lying in $\mathcal{P}_0, \mathcal{N}_0$ and $\mathcal{T}_0$}\label{sec:zerolevel}

It is well-known that the signature of a positive knot is non-positive ~\cite[Corollary 3.4]{CGompf}. More generally we shall see that  (the signs of) many concordance invariants obstruct being slice in a positive definite manifold.  In this section we show that membership in $\mathcal{P}_0$ is obstructed by the signs of classical signatures as well as the sign of the $\tau$-invariant and $s$-invariant. We also see that the slicing obstructions of Donaldson, Fintushel-Stern  and Hedden-Kirk also obstruct membership in $\mathcal{P}_0$.

If $K$ is a knot in $S^3$, $V$ is a Seifert matrix for $K$ and $\omega$ is a complex number of norm $1$, then recall the Levine-Tristram $\omega$-signature of $K$, $\sigma_K(\omega),$  is the signature of
$$
(1-\omega)V+(1-\overline{\omega})V^T.
$$
However, for $\omega$ equal to a root of the Alexander polynomial of $K$, we redefine $\sigma_K(\omega)$ to be the average of the two limits $\displaystyle\lim_{\alpha\rightarrow \omega^{\pm}} \sigma_K(\alpha)$. The resulting function, $\sigma_K:S^1\to \Z$, we shall call the \textbf{Levine-Tristram signature function} of $K$. This function is a concordance invariant. If $p$ is a prime the signatures corresponding to $\omega^j$ where $\omega=\exp(\frac{2\pi i}{p^r})$ are called the (Tristram) $p^r$-signatures of $K$.

The following generalizes Theorem 3.16 and Lemma 4.3 of ~\cite{CGompf}

\begin{prop}\label{prop:easy2fold} If $K\in \mathcal{P}_0$  then the Levine-Tristram signature function of $K$ is non-positive.
Moreover, for a prime power $p^r$,  if,  in addition, the $p^r$-signatures of $K$ are zero, then the $p^r$-fold cyclic cover of $S^3$ branched over $K$ bounds a compact $4$-manifold $\widetilde{V}$ whose intersection form is positive definite and for which $H_1(\widetilde{V};\mathbb{Z}_p)=0$.
\end{prop}

Thus if  $K\in \cB_0$ then its signature function is identically zero. Since it is known that knot signature function detects elements that are of infinite order in Levine's algebraic concordance group we have:

\begin{cor}\label{cor:Zofiniteorder} If $K\in \cB_0$ then $K$ has finite order in the algebraic concordance group. Moreover, for any prime power $p^r$, the $p^r$-fold cyclic cover of $S^3$ branched over $K$ bounds two compact $4$-manifolds $\widetilde{V}_{\pm}$ whose intersection forms are (respectively) $\pm$-definite and for which $H_1(\widetilde{V}_{\pm};\mathbb{Z}_p)=0$.
\end{cor}

\begin{cor}\label{cor:CmodZzero}  The knot signatures corresponding to the different prime roots of unity yield an epimorphism
$$
\frac{\C}{\cB_0}\onto \Z^{\infty}.
$$
\end{cor}

\begin{proof}[Proof of Proposition~\ref{prop:easy2fold}] Since $K\in \mathcal{P}_0$, $K$ bounds a slice disk $\Delta$ in a manifold $V$ as in Definition~\ref{def:positive}. Now we mimic the proof of ~\cite[Theorem 3.7]{CLick}. Let $d=p^r$ be a prime power and let $\Sigma$ denote the $d$-fold cyclic cover of $S^3$ branched over $K$, which is well known to be a $\Z_p$-homology sphere ~\cite[Lemma 4.2]{CG1}. Since $\Delta$ is disjoint from a basis for $H_2(V)$, it represents zero in $H_2(V,\partial V)$. It follows that $H_1(V-\Delta)\cong\Z$, generated by the meridian. Thus the $d$-fold cyclic cover of $V$ branched over $\Delta$, denoted  $\widetilde{V}$, is defined and has boundary $\Sigma$. Since $H_1(V;\Z_p)=0$, it follows from the proof of ~\cite[Lemma 4.2]{CG1} that $H_1(\widetilde{V};\Z_p)=0$ . Thus $\beta_1(\widetilde{V})=0=\beta_3(\widetilde{V})$.

To compute the signature of $\widetilde{V}$ we make $V$ into a closed $4$-manifold and use the $G$-signature theorem. Let $(B^4,F_K)$ be the $4$-ball together with a Seifert surface for $K$ pushed into its interior. Let
$$
(Y,F)=(V,\Delta)\cup (-B^4,-F_K)
$$
be the closed pair, let $\widetilde{W}$ denote the $d$-fold cyclic branched cover of $(B^4,F_K)$, and let $\widetilde{Y}$ be the $d$-fold cyclic branched cover of $(Y,F)$. Note that $\Z_d$ acts on $\widetilde{V}$, $\widetilde{Y}$, and $\widetilde{W}$ with $V$, $Y$ and $B^4$ respectively as quotient. Choose a generator $\tau$ for this action. Let $H_i(\widetilde{Y},j;\mathbb{C})$, $0\leq j<d$, denote the $\exp(\frac{2\pi i j}{d})$-eigenspace for the action of  $\tau_*$ on $H_i(\widetilde{Y};\mathbb{C})$; let $\beta_i(\widetilde{Y},j)$ denote the rank of this eigenspace, and let $\chi(\widetilde{Y},j)$ denote the alternating sum of these ranks (similarly for $\widetilde{V}$ and $\widetilde{W}$). Let $\sigma(\widetilde{Y},j)$ denote the signature of the $\exp(\frac{2\pi i j}{d})$-eigenspace of the isometry $\tau_*$ acting on $H_2(\widetilde{Y};\mathbb{C})$ (similarly for $\widetilde{V}$ and $\widetilde{W}$). By a lemma of Rochlin, using the $G$-signature theorem ~\cite{Rok1}\cite[Lemma 2.1]{CG1}, since $\widetilde{Y}$ is closed and $[F]\cdot[F]=0$,
$$
\sigma(\widetilde{Y},j)=\sigma(Y).
$$
Since $\widetilde{Y}=\widetilde{V}\cup -\widetilde{W}$ glued along the rational homology sphere $\Sigma$, this translates to
$$
\sigma(\widetilde{V},j)-\sigma(\widetilde{W},j)=\sigma(V)-\sigma(B^4).
$$
Since the intersection form of $V$ is, by assumption, positive definite, $\sigma(V)=\beta_2(V)$, so
$$
\sigma(\widetilde{V},j)-\sigma(\widetilde{W},j)=\beta_2(V).
$$
Hence
\begin{equation}\label{eq:sigs1}
\sigma(\widetilde{W},j)=\sigma(\widetilde{V},j)-\beta_2(V).
\end{equation}

Consider the covering space $\widetilde{V}-\widetilde{\Delta}\to V-\Delta$. By ~\cite[Proposition 1.1]{Gi5}, for any $0\leq j<d$,
\begin{equation}\label{eq:eulerbar}
\chi(V-\Delta)=\chi(\widetilde{V}-\widetilde{\Delta},j).
\end{equation}
Since $\tau$ acts by the identity on $H_0(\widetilde{V}-\widetilde{\Delta})$,  $\beta_0(\widetilde{V}-\widetilde{\Delta},0)=1$ and, if $j\neq 0$, $\beta_0(\widetilde{V}-\widetilde{\Delta},j)=0$. Since $\beta_1(V-\Delta)=1$, $\beta_1(\widetilde{V}-\widetilde{\Delta})\geq 1$. On the other hand, since $\beta_1(\widetilde{V})=0$, we must have $\beta_1(\widetilde{V}-\widetilde{\Delta})=1$ generated by a meridian. Since $\tau$ acts by the identity on the first homology of this meridian, $\beta_1(\widetilde{V}-\widetilde{\Delta},0)=1$ and, if $j\neq 0$, $\beta_1(\widetilde{V}-\widetilde{\Delta},j)=0$. Since $\widetilde{V}$ is obtained from $\widetilde{V}-\widetilde{\Delta}$ by adding a $2$-handle along a circle of infinite homological order,
$$
\beta_2(\widetilde{V}-\widetilde{\Delta},j)=\beta_2(\widetilde{V},j).
$$
For the same reason, since
$\beta_3(\widetilde{V})=0$, $\beta_3(\widetilde{V}-\widetilde{\Delta},j)=0$. Similarly, $H_2(V-\Delta)\cong H_2(V)$ and $H_3(V-\Delta)\cong H_3(V)=0$. Thus equation ~(\ref{eq:eulerbar}) becomes
\begin{equation}\label{eq:rankH2}
\beta_2(V)=\beta_2(\widetilde{V},j), ~~\text{and so} ~~ d\beta_2(V)=\beta_2(\widetilde{V}).
\end{equation}
Combining this with equation ~(\ref{eq:sigs1}) we have
\begin{equation}\label{eq:sigs2}
\sigma(\widetilde{W},j)=\sigma(\widetilde{V},j)-\beta_2(\widetilde{V},j).
\end{equation}
Thus $\sigma(\widetilde{W},j)$ is non-positive.

But it is known that, if $j\neq 0$, then $\sigma(\widetilde{W},j)$ is a \textbf{ $p^r$-signature of $K$} ~\cite{Vi1}\cite[Chapter 12]{Gor1}, that is
$$
\sigma(\widetilde{W},j)=\sigma_{\omega^j}(K),
$$
where $\omega=\exp(\frac{2\pi i}{d})$. Since the roots of unity, as $p^r$ varies, are dense in the circle, this implies that the entire signature function of $K$ is non-positive.

Additionally, from ~(\ref{eq:sigs2}), we see also that $\widetilde{V}$ is positive definite if and only if $\sigma(\widetilde{W},j)=0$ for each $j$. Thus if all of the $p^r$signatures of  $K$ are zero, then $\widetilde{V}$ is positive definite.
\end{proof}

Even for topologically slice knots, membership in $\mathcal{P}_0$ is often obstructed by the theorems of Donaldson ~\cite{Don1}, Fintushel-Stern ~\cite{FiSt2} and Ozsv\'ath-Szabo ~\cite{OzSz2}. In this regard the following elementary observation is useful (this result is almost the same as ~\cite[Lemma 2.10]{CGompf}). Recall that if $K$ is an oriented knot and $p/q\in \Q$ is non-zero then $Y=S^3_{p/q}(K)$, the $p/q$-framed Dehn surgery on $K$, is a rational homology $3$-sphere with $H_1(Y)\cong \Z_p$ via a canonical map sending the meridian to $1$. 

\begin{prop}\label{prop:1surgery} If $K\geq J$ then for any non-zero $p/q\in \Q$,  $Y=S^3_{p/q}(K)\coprod -S^3_{p/q}(J)$  bounds a compact $4$-manifold $W$ with intersection form isomorphic to $\oplus\langle 1\rangle$, for which there exist canonical isomorphisms $H_1(\partial_{\pm}W)\cong H_1(W)\cong \Z_p$. In particular if  $K\in \mathcal{P}_0$ (respectively $\cN_0$) then both the $+1$-framed surgery on $K$ and the  $-1$ surgery on $K$ bound compact $4$-manifolds with  positive definite (respectively negative definite) diagonalizable intersection form and $H_1=0$.
\end{prop}
\begin{proof} Suppose $K\geq J$  via $V$ (as in Definition~\ref{def:geq} for $n=0$). Then doing ``Dehn surgery cross $[0,1]$'' on the annulus in $V$ gives the desired manifold $W$ (more details are given in ~\cite[Lemma 2.10]{CGompf}). Here, since we have assumed $\pi_1(V)=0$, we get the extra $H_1$-isomorphism property that was not present in ~\cite{CGompf}.

For the second statement, take $J=U$ and let $p=\pm 1$. We remark that the $+1$-framed surgery on \textit{any} knot bounds such a positive definite manifold, so really it is only the  statement for $-1$ that has content.  
\end{proof}

As a consequence, we will show  in Corollary~\ref{cor:dinvariantssurgeries} that the signs of the Ozsv\'ath-Szabo d-invariants associated to the $\pm 1$-surgeries on a knot obstruct membership in $\cP_0$ and $\cN_0$.

\begin{ex}\label{ex:Gompf} Let $K=WH^-(LHT)$ where $LHT$ denotes the left-handed trefoil knot. Since $K$ has Alexander polynomial one it is topologically slice ~\cite{FQ}. Thus $K\in \mathcal{T}\cap\mathcal{N}_0$ by Proposition~\ref{prop:easyexs}. But $K\notin \cP_0$  because, by ~\cite[Corollary 2.5]{Go1}, $-1$-surgery on $K$ does not bound a $4$-manifold with positive definite intersection form as required by Proposition~\ref{prop:1surgery}. Hence $K\notin\mathcal{T}_0$
\end{ex}

\begin{ex}\label{ex:pqr} Suppose $p<0,q>0,r>0$ are odd and $pq+qr+rp=-1$. Then the pretzel knot $K(p,q,r)$ has Alexander polynomial $1$ and hence is topologically slice. As long as no product of two of $p,q,r$ is $-1$, it is of infinite order in the smooth concordance group ~\cite[Corollary 4.3]{CLick}. Moreover the $2$-fold branched cover, $\Sigma_K$, is the Brieskorn sphere $\Sigma(|p|,q,r)$ (with its orientation as the boundary of the canonical negative definite resolution)~\cite[140-141]{CLick}. Then, by ~\cite[Thms. 10.1, 10.4]{FiSt2}, $\Sigma_K$ cannot bound a positive definite $4$-manifold as in the conclusion of Proposition~\ref{prop:easy2fold}. Thus $K(p,q,r)\notin \mathcal{T}_0$. Alternatively, $\Sigma(p,q,|r|)$ bounds its canonical $1$-connected negative definite plumbing ~\cite[137-141]{CLick}, $Y$, whose intersection form is \emph{not diagonalizable} (this is proved in ~\cite[proof of Prop. 3.1,page 11-12]{GJ}). Hence $K(p,q,r)\notin \mathcal{P}_0$ since $W\cup -Y$ would violate Donaldson's theorem.
\end{ex}

In ~\cite{Endo} Endo showed that a certain infinite subset of the family of knots in Example~\ref{ex:pqr} is linearly independent in $\T$. In fact  Endo's argument (using the full strength of techniques of Furuta and Fintushel-Stern), together with our Proposition~\ref{prop:easy2fold}, shows the following:

\begin{thm}\label{thm:Endo} The family of topologically slice pretzel knots considered by Endo generates a
$$
\Z^\infty\subset \T/\T_0.
$$
\end{thm}
\begin{proof} The claim is that Endo's original proof (together with our Proposition~\ref{prop:easy2fold} for $p=2$) proves this stronger result. Endo's family is a specific sequence of pretzel knots, $K_k$, $k=1,...$,  as in Example~\ref{ex:pqr}. Following Endo, we proceed by contradiction. Suppose that some non-trivial linear combination, $K$, of such knots were to lie in $\T_0$. By taking the concordance inverse we may assume that if $m$ is the largest value for which $K_m$ occurs, then it occurs with a positive coefficient. Since $K\in \T_0$,  $K\in\cP_0$. The $2$-fold branched cover, $\Sigma_K$, is a connected sum of the corresponding Brieskorn homology spheres.  By Proposition~\ref{prop:easy2fold} for $p=2$, $\Sigma_K$ is the boundary of a compact $4$-manifold $\widetilde{V}$ whose intersection form is positive definite and for which $H_1(\widetilde{V};\mathbb{Z}_2)=0$. Endo's argument shows (using work of Furuta) that this is a contradiction in the case that $\beta_2(\widetilde{V})=0$, but it was known that the results of Furuta employed in the proof hold in our more general context (see ~\cite[Page 340]{Fur}, ~\cite[Thm. 1.1 and remarks below Thm. 1.2]{FiSt2}~\cite[proof of Thm. 5.1]{FiSt3}.
\end{proof}

More generally, for any knot $K\in \cP_0$ with vanishing $2^r$-signatures, by  Proposition~\ref{prop:easy2fold}, the $2^r$-fold cyclic cover of $S^3$ branched over $K$ bounds a compact $4$-manifold $\widetilde{V}$ whose intersection form is positive definite and for which $H_1(\widetilde{V};\mathbb{Z}_2)=0$. Therefore the obstruction of Hedden-Kirk (Fintushel-Stern, Furuta) vanishes.

Additionally, it follows immediately from  work of Ozsv\'ath-Szabo ~\cite[Thm. 1.1]{OzSz2} that

\begin{prop}\label{prop:tau}(Ozsv\'ath-Szabo) If $K\in \mathcal{P}_0$ then $\tau(K)\geq 0$.
\end{prop}

\begin{cor}\label{cor:taunogood} If $K\in \cB_0$ then $\tau(K)=0$.
\end{cor}

Then, regards Jennifer Hom's $\epsilon$-invariant we have:

\begin{prop}\label{prop:epsilon} If $K\in \mathcal{B}_0$ then $\epsilon(K)=0$.
\end{prop}
\begin{proof}  Note that both $K_{2,1}$, the $(2,1)$-cable of $K$,  and $K_{2,-1}$, the $(2,-1)$-cable of $K$, may be viewed as satellites of $K$ with patterns $U_{2,1}$ and $U_{2,-1}$. Since the latter are  trivial knots, by Corollary~\ref{cor:satellites}, both $K_{2,1}$ and $K_{2,-1}$ lie in $\cB_0$. By Corollary~\ref{cor:taunogood},  $\tau(K)=\tau(K_{2,1})=\tau(K_{2,-1})=0$. It then follows from ~\cite[Theorem 5.2]{Hom}
that $\epsilon(K)=0$.
\end{proof}

As regards the Rasmussen $s$-invariant of knot concordance, it  follows quickly  from  recent work of Kronheimer-Mrowka ~\cite[Cor. 1.1]{KrMr1} that

\begin{prop}\label{prop:s} If $K\in \mathcal{P}_0$ then $s(K)\geq 0$.
\end{prop}

\begin{proof} Suppose that $K$ bounds the slice disk $\Delta$ in the $0$-positon $V$.  We think of $\Delta$ as representing a class in $\pi_2(V,K)$. Consider the exact sequence
$$
\pi_2(S^3,K)\overset{j_*}\rightarrow\pi_2(V,K)\overset{i_*}\rightarrow\pi_2(V,S^3)\cong H_2(V,S^3).
$$
Since $\pi_1(V)=\pi_1(S^3)=0$, by the relative Hurewicz theorem, $H_2(V,S^3)\cong \pi_2(V,S^3)$. Since, by definition, $i_*([\Delta])=0$ in $H_2(V,S^3)$, $i_*([\Delta])=0$ in $\pi_2(V,S^3)$. Hence $[\Delta]$ is in the image of $j_*$. Now,  by ~\cite[Cor. 1.1]{KrMr1}, $s(K)\geq 0$.
\end{proof}

\begin{cor}\label{cor:svanish} If $K\in \cB_0$ then $s(K)=0$.
\end{cor}

\section{Relations with the $n$-solvable filtration and von Neumann signature defect obstructions}\label{sec:relationsn-solv}

We show that the $n$-positive filtration and the $n$-negative filtration are, essentially, refinements of the $n$-solvable filtration of Cochran-Orr-Teichner ~\cite{COT}. This is a philosophically important point but also  allows us to use higher-order signatures to obstruct membership in $\cB_n$.

Before discussing the (simple) connection between these filtrations, the reader might benefit from being aware of a slight paradigm shift leading to an equivalent definition of $n$-positivity (and $n$-negativity), which looks more like the original definition of the $n$-solvable filtration. Note that if $K$ is slice in a manifold $V$ then the boundary of $V\setminus \Delta$, which we denote by $W$, may be identified with the zero-framed surgery on $K$, $M_K$. Hence conditions on $V\setminus \Delta$  and $K$ may be re-interpreted as conditions on $W$ and its boundary $M_K$. 

\begin{defn}\label{def:seconddef} (Alternative to Definition~\ref{def:positive}) We say that a knot is an element of $\mathcal{P}_{n}$, and said to be \textbf{n-positive}, if the zero-framed surgery $M_K$ bounds a compact, oriented, connected, smooth $4$-manifold $W$ such that
\begin{itemize}
\item [1.] $H_1(M_K;\Z)\to H_1(W;\Z)$ is an isomorphism and $\pi_1(W)$ is normally generated by a meridian of $K$;
\item [2.] $H_2(W;\Z)$ has a basis consisting of connected compact oriented surfaces, $\{S_i~|~1\leq i\leq r\}$, disjointly embedded in $W$ with normal bundles of Euler class $1$.
\item [3.] for each $i$, $\pi_1(S_i)\subset \pi_1(W)^{(n)}$.
\end{itemize}
Thus in particular the intersection form on $H_2(W)$ is isomorphic to the standard diagonal form $\oplus \langle +1\rangle$ (this includes the case that $H_2(W)=0$). Then we say that $M_K$ is \textbf{n-positive} \textbf{via} $W$ (compare ~\cite[Def.8.7]{COT}), and  that $W$ is an \textbf{n-positon for $M_K$}. Similarly by changing the Euler class to $-1$, we define $\mathcal{N}_n$, the \textbf{n-negative} knots. In this case $W$ is called an \textbf{n-negaton for $M_K$}. 
\end{defn}

\begin{prop}\label{prop:defsequiv} Definition~\ref{def:positive} is equivalent to Definition~\ref{def:seconddef}
\end{prop}

\begin{proof} Suppose that $K$ satisfies Definition~\ref{def:positive} via $V$. Let $W$ be the exterior of the slice disk in $V$. Then $\partial W=M_K$, $H_1(M_K)\cong H_1(W)\cong \Z$ and $\pi_1(W)$ is normally generated by the meridian of $K$. Moreover $H_2(W)\cong H_2(V)$ which by hypothesis has a basis of disjoint embedded surfaces $S_i$ such that $\pi_1(S_i)\subset \pi_1(W)^{(n)}$. As previously observed, the intersection matrix with respect to this basis is the identity matrix. Equivalently, for each $i$, the Euler class of the normal bundle of $S_i$ is $+1$. Hence $K$ is $n$-positive via $W$ as required by Definition~\ref{def:seconddef}.

Conversely suppose $K$ satisfies Definition~\ref{def:seconddef} and $M_K$ is $n$-positive via $W$. Add a $2$-handle along the meridian of $K$ to yield at  a smooth $4$-manifold $V$ such that, $\partial V=S^3$, $\pi_1(V)=0$, and the inclusion of $W$ into $V$ induces an isometry of the intersection form. Moreover the cobordism from $M_K$ to $S^3$ given by adding the $2$-handle above, if turned up-side down, is merely the cobordism from $S^3$ to $M_K$ obtained by adding  a $0$-framed $2$-handle along of $K$. Thus the cocore of the original handle is a smoothly embedded slice disk, $\Delta$, for $K\subset \partial V$. From this point of view it is clear that a copy of  $K$ may be viewed as living in $\partial V$ and is slice $V$. Moreover, the collection of surfaces $\{S_i\}$ lying in $W$ is a basis for $H_2(V)$ and is disjoint from $\Delta$. Indeed $V\setminus \Delta$ is diffeomorphic to $W$. Thus this collection also satisfies Definition~\ref{def:positive}.

\end{proof}

Definition~\ref{def:seconddef} is very similar to the definition (given first in ~\cite[Def. 2.3]{CHL5}) of a knot's being $n$-solvable in the sense of Cochran-Orr-Teichner ~\cite{COT}. To relate the $n$-positive filtration to the $n$-solvable filtration, first we define a filtration,  $\{\mathcal{F}_{n}^{odd}\}$, which is slightly larger than the $n$-solvable filtration, $\{\mathcal{F}_{n}\}$, by dropping the ``spin condition'' in the latter. Indeed $3$ slight variations of the $n$-solvable filtration have appeared:
$$
\mathcal{F}_{n}\subset \mathcal{F}_{n}^{odd}\subset \mathcal{F}_{n}^\Q.
$$
The first is the $n$-solvable filtration in ~\cite{COT}. The third was called the \textit{rational} $n$-solvable filtration in ~\cite[Definition 4.1]{COT}.  $\mathcal{F}_{n}^{odd}$, which we define below, is essentially the same as what was called the \textit{integral} $n$-solvable filtration in ~\cite[Definition 3.1]{Cha6}.

\begin{defn}\label{def:nsolvableodd} A knot $K$ is an element of $\mathcal{F}_{n}^{odd}$ if
the zero-framed surgery $M_K$ bounds a compact, connected, oriented, smooth $4$-manifold $W$ such that
\begin{itemize}
\item [1.] $H_1(M_K;\Z)\to H_1(W;\Z)$ is an isomorphism;
\item [2.] $H_2(W;\Z)$ has a basis consisting of connected, compact, oriented surfaces, $\{L_i,D_i~|~1\leq i\leq r\}$,  embedded in $W$, wherein the surfaces are pairwise disjoint except that, for each $i$, $L_i$ intersects $D_i$ transversely once with positive sign; moreover, the $L_i$ have trivial normal bundles;
\item [3.] for each $i$, $\pi_1(L_i)\subset \pi_1(W)^{(n)}$
and $\pi_1(D_i)\subset \pi_1(W)^{(n)}$.
\end{itemize}
A knot $K\in \mathcal{F}_{n.5}^{odd}$ if, in addition,
\begin{itemize}
\item [4.] for each $i$, $\pi_1(L_i)\subset \pi_1(W)^{(n+1)}$.
\end{itemize}
\end{defn}

Recall that a knot is \textbf{$n$-solvable} or lies in $\mathcal{F}_{n}$, if we additionally require that the $D_i$ have trivial normal bundles \cite[Definition 2.3]{CHL5}. This forces the intersection form to be hyperbolic and hence forces $W$ to be spin. All of the results of ~\cite{COT} that obstruct membership in $\mathcal{F}_{n}$, obstruct membership in $\mathcal{F}_{n}^{odd}$ as we shall see. In fact, it was already observed in ~\cite[Section 4]{COT} that these results held for an even larger filtration, called the \emph{rational $n$-solvable filtration}, denoted $\boldsymbol{\mathcal{F}_{n}^\Q}$. The definition of the latter is similar to that above, replacing all occurrences of $\Z$ are replaced by $\Q$ (or see ~\cite[Definition 4.1]{COT}).

Note that the definitions of $K\in \cP_n$ ( Definition~\ref{def:seconddef}), $K\in \mathcal{F}_{n}$, and $K\in\mathcal{F}_{n}^{odd}$ (Definition~\ref{def:nsolvableodd})  are identical except that the intersection form on $H_2(W)$ is required to be, in the first case, a direct sum of $+1$'s, in the second case a direct sum of hyperbolic pairs, and in the third case, a direct sum of matrices of the form
$$
\left(\begin{array}{cc}
0 & 1\\
1 & *
\end{array}\right),
$$
where $*$ is arbitrary. 

We remark, even though it is not necessary for this paper, that there are equivalent alternative  definitions  of $\mathcal{F}_{n}^{odd}$ (and its variants). Specifically:

\begin{defn}\label{def:altdefoddsolv} (Alternative Defintion) $K\in\mathcal{F}_{n}^{odd}$ if 
$K$ bounds a smooth $2$-disk $\Delta$ in a smooth $4$-manifold $V$ such that
\begin{itemize}
\item [1.] $H_1(V)=0$;
\item [2.] $H_2(V;\Z)$ has a basis consisting of connected, compact, oriented surfaces, $\{L_i,D_i~|~1\leq i\leq r\}$,  embedded in $V\setminus \Delta$, wherein the surfaces are pairwise disjoint except that, for each $i$, $L_i$ intersects $D_i$ transversely once with positive sign; moreover, the $L_i$ have trivial normal bundles; 
\item [3.]  $\pi_1(L_i)\subset \pi_1(V-\Delta)^{(n)}$ and $\pi_1(D_i)\subset \pi_1(V-\Delta)^{(n)}$
\end{itemize}
A knot $K\in \mathcal{F}_{n.5}^{odd}$ if, in addition,
\begin{itemize}
\item [4.] for each $i$, $\pi_1(L_i)\subset \pi_1(V-\Delta)^{(n+1)}$
\end{itemize}
\end{defn}

Finally:

\begin{prop}\label{prop:positiveinratnsolv} For any integer $n$,
$$
\mathcal{P}_n\subset \mathcal{F}_n^{odd}.
$$
\end{prop}

The key to the  proof is the simple fact that  $<+1>\oplus <-1>$ is congruent to the matrix 
$$
\left(\begin{array}{cr}
0 & 1\\
1 & -1
\end{array}\right).
$$
More geometrically, the idea is that if a knot is slice in a smooth manifold $V$ that is homeomorphic to a punctured $\#_m (\mathbb{C}P(2))$ then, by connect-summing with copies of $-\mathbb{C}P(2)$ away from the slice disk, we see that $K$ is also slice in $V\#_m (-\mathbb{C}P(2))$. The latter is known to be homeomorphic to a punctured $\#_m S^2\tilde{\times}S^2$.

\begin{proof}   Suppose $K$ satisfies Definition~\ref{def:seconddef} via $W$ whose intersection form is $\oplus_m \langle +1\rangle$ with basis $\{e_i\}$ represented by $\{S_i~|~ 1\leq i\leq m\}$. Then let $W'=W\#_m (-\mathbb{C}P(2))$, performing the connected sums far away from the $S_i$. Note that the fundamental group is unchanged. Let $P_i$ be the embedded $\mathbb{C}P(1)$ contained in the $i^{th}$ copy of $-\mathbb{C}P(2))$. Tube $P_i$ to $S_i$ (avoiding the other $S_j$ and $P_j$) and call the resulting embedded surface $L_i$.  It will have self-intersection zero and hence trivial normal bundle. Let $D_i$ be a push-off of $S_i$. Then $D_i$ intersects $L_i$ precisely once. It follows that $W'$ together with the collection $\{L_i,D_i\}$ satisfies Definition~\ref{def:nsolvableodd}. Hence $K\in \mathcal{F}_n^{odd}$.

\end{proof}

\begin{cor}\label{cor:positiveinratnsolv} For any integer $n$,
$$
\cB_n\subset\langle\mathcal{P}_n\rangle=\langle\mathcal{N}_n\rangle=\langle\mathcal{P}_n\cup\mathcal{N}_n\rangle\subset \mathcal{F}_n^{odd}\subset \mathcal{F}_n^{\Q}.
$$
\end{cor}
\begin{proof} Here the brackets mean ``the subgroup generated by''. The last inclusion follows immediately from the definitions, as do the equalities.  If $K\in \mathcal{N}_n$ then $-K\in \mathcal{P}_n$ so  $-K\in \mathcal{F}_n^{odd}$ by Proposition~\ref{prop:positiveinratnsolv}. Thus $K\in \mathcal{F}_n^{odd}$ since the latter is a subgroup.
\end{proof}

\begin{cor}\label{cor:ponealgslice} If $K\in \mathcal{P}_1$ (or $K\in \mathcal{N}_1$) then $K$ is algebraically slice.
\end{cor}
\begin{proof} If $K\in \mathcal{P}_1$ then $K\in \mathcal{F}_1^{odd}$. By ~\cite[Theorem 4.4]{COT} the Blanchfield form on the rational Alexander module of $K$ has a self-annihilating submodule (a Lagrangian). It follows that $K$ is an algebraically slice knot.
\end{proof}

Because of the close relationship between the $n$-positive filtration and (variations of) the $n$-solvable filtration, there exist \emph{signature-defect} obstructions that can assist in determining whether or not a given knot lies in a particular term of $\mathcal{P}_{*}$ or $\cB_{*}$. Given a closed, oriented 3-manifold $M$, a discrete group $\G$, and a representation $\phi : \pi_1(M)
\to \G$, the \textbf{von Neumann
$\boldsymbol{\rho}$-invariant}, $\rho(M,\phi)$, was defined by Cheeger and Gromov.

\begin{thm}\label{thm:generalsignaturesobstruct}  Suppose $K\in \mathcal{P}_{n}$, so $K$ bounds a slice disk $\Delta$ in an $n$-positon $V$. Suppose $\G$ is a poly-(torsion-free-abelian) group and 
$$
\psi:\pi_1(V\setminus \Delta)\to \G,
$$
is a homomorphism  whose restriction to $\pi_1(\partial(V\setminus\Delta))\cong \pi_1(M_K)$ is denoted  $\phi$. 
Then
$$
\rho(M_K,\phi)\leq 0;
$$
and if, moreover, $\phi(\pi_1(V\setminus\Delta)^{(n)})=0$, then
$$
\rho(M_K,\phi)=0.
$$
\end{thm}
\begin{proof} Let $W=V\setminus\Delta$. Recall that
$$
\rho(M_K,\phi)= \sigma^{(2)}(W,\psi)-\sigma(W)
$$
where $\sigma^{(2)}(W,\psi)$ is the von Neumann signature  of $W$ corresponding to $\psi$ (also called the $L^{(2)}$-signature) (see ~\cite[Section 5, Lemma 5.9]{COT}. Since the intersection form on $W$ is positive-definite, $\sigma(W)=\beta_2(W)$. But it is known that 
$$
|\sigma^{(2)}(W,\psi)|\leq \beta_2(W)
$$
(see for example ~\cite[Lemma 2.7]{Cha3}). The first part of the theorem follows immediately. 

Now suppose $\psi(\pi_1(V\setminus\Delta)^{(n)})=0$. Since $K\in \cP_0$, by Corollary~\ref{cor:positiveinratnsolv}, $\K\in \mathcal{F}_{n}^\Q$. Hence $\K\in \mathcal{F}_{n-.5}^\Q$. Then,  by ~\cite[Theorem 4.2]{COT},  $\rho(M_L,\phi)=0$.
\end{proof}

Another major result of ~\cite{COT}  is that if $K $is a  slice knot then certain so-called higher-order Alexander modules of $K$ have submodules that are self-annihilating with respect to a higher-order Blanchfield linking form. Specifically if $\G$ is a poly-(torsion-free-abelian) group  then $\Z\G$ is an Ore domain with classical (skew) field of fractions $\mathcal{K}$. Suppose $\mathcal{R}$ is a classical localization of  $\Z\G$ where $\Z\G\subset \mathcal{R}\subset \mathcal{K}$. If $\phi:\pi_1(M_K)\to\G$ is a coefficient system then $H_1(M_K; \mathcal{R})$ is defined and is called a higher-order Alexander module (see ~\cite[Section 2]{COT}). Moreover there is a linking form defined on this higher-order Alexander module that takes values in $\mathcal{K}/\mathcal{R}$.

\begin{thm}\label{thm:hyperbolicity}  Suppose $n\geq 1$ and $K\in \mathcal{P}_{n}$, so $K$ bounds a slice disk $\Delta$ in an $n$-positon $V$. Let $W=V\setminus \Delta$. Suppose $\phi$ above extends to  $\pi_1(W)$ such that $\psi(\pi_1(W)^{(n-1)})=0$.  Let $P$ be the kernel of
$$
j_*:H_1(M_K; \mathcal{R})\to H_1(W; \mathcal{R}).
$$
Then $P\subset P^\perp$ with respect to the higher-order linking form. If $\mathcal{R}$ is a PID then $P=P^\perp$.
\end{thm}

\begin{proof} Since $K\in \cP_0$, by Corollary~\ref{cor:positiveinratnsolv}, $\K\in \mathcal{F}_{n}^\Q$. Then the claimed result follows from ~\cite[Theorem 4.4]{COT}. Note that in the proof of that theorem, the fact that $\mathcal{R}$ was a PID was not used to prove that $P\subset P^\perp$.

\end{proof}

Note that  Theorem~\ref{thm:hyperbolicity} may be applied to the abelianization map $\phi:\pi_1(M_K)\to\Z=\Gamma$,  with $\mathcal{R}=\Q[t,t^{-1}]$ and with $n=1$ to yield Corollary~\ref{cor:ponealgslice}.

\section{d-invariant obstructions to membership in $\mathcal{P}_1, \mathcal{N}_1, \cB_1$ and $\mathcal{T}_1$}\label{sec:firstlevel}

In this section we prove we prove that (the signs of) the Ozsv\'ath-Szabo correction terms for the prime-power cyclic branched covers of a knot obstruct that knot's lying in $\cN_1$ (respectively, in $\P_1$). First we review the known result that we will generalize, namely that these corrections terms obstruct a knot,s being a slice knot.

Recall that if  $K$ is a knot then any prime-power cyclic cover of $S^3$ branched over $K$ is a $\Q$-homology sphere $Y$  ~\cite{CG1}. There is a symmetric, non-singular \textit{linking form} defined on the torsion subgroup of the first homology of any oriented $3$-manifold: 
$$
\ell : H_1(Y) \times H_1(Y)\to \Q  ~\text{mod} ~ \Z.
$$
Also recall that, given a rational homology sphere $Y$ and a $\spinc$ structure $\s$ on $Y$,  Ozsv\'ath and Szab\'o have defined  the so-called \emph{correction term} $d(Y,\s)\in \Q$ ~\cite[Definition 4.1]{OzSz1}. They showed that the correction terms are invariants of the $\spinc$-rational homology cobordism class. More generally, if $Y$ bounds a negative-definite $4$-manifold $X$, and if $\s\in\spinc(Y)$ extends to $\t\in\spinc(X)$, the correction term satisfies~\cite[Theorem 9.6]{OzSz1}
\begin{equation}\label{eq:OSinequality}
	c_1(\t)^2 + \beta_2(X) \leq 4\,d(Y,\s)
\end{equation}
There is a similar inequality if $X$ is positive definite. 
\begin{equation}\label{eq:OSinequality2}
	c_1(\t)^2 - \beta_2(X) \geq 4\,d(Y,\s)
\end{equation}
The meaning of $c_1(\t)^2$ is explained more thoroughly later in this section.

These invariants can be combined to yield obstructions to a knots being a slice knot, which we now review. It is these obstructions that we will generalize. Suppose $K$ is a slice knot. In that case it is known that the corresponding cover of $B^4$ branched over the slice disk is a $\Q$-homology ball $\widetilde{V}$ and that the linking form is \textbf{metabolic} ~\cite{CG1} . The latter means that there is a \textbf{metabolizer} for the linking form, that is, a subgroup $G$, such that $G=G^\perp$ with respect to $\ell$.   Thus, for any $\t$, $\beta_2(\widetilde{V})=0=c_1(\t)^2$ (so its intersection form is both positive and negative definite), and it follows from ~\ref{eq:OSinequality} and ~\ref{eq:OSinequality2} that $d(Y,\s)=0$. Various authors have observed that this line of thinking leads to the meta-statement: \textit{d-invariants of prime-power branched covers obstruct being a slice knot} ~\cite{OwSt,GRSt,JaNa,GJ}. Specifically, by this we  mean:

\begin{thm} ~\cite[Prop. 4.1, Cor. 4.2]{OwSt}  \label{thm:dinvsslice} If $K$ is a slice knot and $Y$ is the $p^r$-fold cyclic  cover of $S^3$ branched over $K$, then there is a metabolizer $G<H_1(Y)$  for the $\Q/\Z$-linking form on $H_1(Y)$; and there exists a $\spinc$ structure $\mathfrak{s}_0$ on $Y$  such that $d(Y,\mathfrak{s}_0+\hat{z})=0$ for all $z\in G$, where $\hat{z}$ is the Poincare dual of $z$. Moreover  if $X$ is the $p^r$-fold cyclic  cover of $B^4$ branched over the slice disk for $K$, then $G$ may be taken to be the kernel of
$$
j_*:H_1(Y)\to H_1(X).
$$
Furthermore  we may take $\mathfrak{s}_0$ to be a $\spinc$ structure corresponding to a spin structure on $Y$.
\end{thm} 

In this section we extend this result to the meta-statement: \textit{(the signs of) d-invariants of prime-power branched covers of $K$ obstruct membership in $\cP_1$ (or in $\cN_1$}) . By this we  mean:

\begin{thm}\label{thm:onepositivedinvariants} If $K\in\cP_1$  and $Y$ is the $p^r$-fold cyclic  cover of $S^3$ branched over $K$, then there is a metabolizer $G<H_1(Y)$ for the $\Q/\Z$-linking form on $H_1(Y)$; and there is a $\spinc$ structure $\mathfrak{s}_0$ on $Y$  such that $d(Y,\mathfrak{s}_0+\hat{z})\leq 0$ for all $z\in G$, where $\hat{z}$ is the Poincare dual of $z$. Moreover  if $X$ is the $p^r$-fold cyclic  cover of the $1$-positon $V$ branched over the slice disk for $K$, then $G$ may be taken to be the kernel of
$$
j_*:H_1(Y)\to H_1(X).
$$
Furthermore we may take $\mathfrak{s}_0$ to be a $\spinc$ structure corresponding to a spin structure on $Y$.  Furthermore, these correction terms will be even integers.
\end{thm}

Before proving Theorem~\ref{thm:onenegativedinvariants}, we review necessary material about $\spinc$-structures. Experts can skip to the proof. Most of what is reviewed her can be found in ~\cite[pp. 55-57]{GompfStip}\cite[pp. 387-390]{Scor}.

Suppose that $Y$ is an oriented $\Q$-homology $3$-sphere. The set of $\spinc$ structures on $Y$  (respectively on a $4$-dimensional oriented manifold $X$ ) is a non-empty affine set in one-to-one correspondence with the second cohomology group of $Y$ (respectively $X$).  If $\t\in \spinc (X)$ and $\a\in H^2(X)$ then we use the notation $\t+\a$ to denote the element of $\spinc (X)$ obtained from  the action of $\a$ on $\t$.  If $\partial X=Y$,  given $\t\in\spinc(X)$, one may \textit{restrict} to get a $\spinc$ structure $\t|_Y\in\spinc(Y)$. Let $j:Y\to X$ denote the inclusion map.  The action of $H^2(\cdot)$ on $\spinc(\cdot)$ is equivariant with respect to the restriction map in the sense that
$$(\t+\a)|_Y = \t|_Y+j^\ast(\a).$$
This information is encoded by the commutativity of the left-hand side of  Diagram (\ref{diag:spincstuff}) below. 
\begin{equation}\label{diag:spincstuff}
\begin{diagram}
\node{H^2(Y)} \arrow{e,t}{\beta\to \s+\beta}\node{\spinc(Y)}\arrow{e,t}{c_1}
 \node{H^2(Y)}\arrow{e,t}{P.D.}\node{H_1(Y)}
\\
\node{H^2(X)} \arrow{e,t}{\a\to \t+\a}\arrow{n,r}{j^*} \node{\spinc(X)}\arrow{n,r}{|_Y}  \arrow{e,t} {c_1} \node{\mathcal{C}_X\subset H^2(X)}\arrow{n,r}{j^*}\arrow{e,t}{P.D.} \node{H_2(X,Y)}\arrow{n,r}{\partial_*}
\end{diagram}
\end{equation}
In addition, any $\spinc$ structure has a \textbf{first Chern class}, which, for a $4$-manifold takes values in the set, $\mathcal{C}_X$, of \textbf{characteristic} elements of $H^2(X)$ (defined below).  The naturality property of the first Chern class is encoded by the commutativity of the middle square of  Diagram (\ref{diag:spincstuff}). The map $c_1$ is surjective (and injective if $H^2(X)$ has no $2$-torsion). The first Chern class is affected by the above action according to the formula
\begin{equation}\label{eq:chernformula}
c_1(\t+\a)= c_1(\t)+2\a 
\end{equation}

Moreover any $\spin$ structure induces a unique $\spinc$-structure. Finally any $\spin$ structure on $Y$ induces a $\spinc$-structure on $Y$ that extends to (is the restriction of) some $\spinc$-structure on $X$.

If $X$ is a compact oriented $4$-manifold with boundary $Y$ and $TH_2$ denotes the torsion subgroup of $H_2$, then the composition
\begin{equation}\label{eq:defintersform}
H_2(X)\overset{\pi_*}{\rightarrow}H_2(X,Y)\overset{P.D.}{\rightarrow}H^2(X)\overset{\kappa}{\rightarrow}\Hom (H_2(X),\Z)
\end{equation}
induces a symmetric  bilinear form:
$$
H_2(X)/TH_2(X)\otimes H_2(X)/TH_2(X)\to \Z
$$
that we call the \textbf{intersection form of $X$}, denoted $x\cdot y$. By duality there is an equivalent form
\begin{equation}\label{eq:QX}
Q_X:H^2(X,Y)\otimes H^2(X,Y)\to\Z
\end{equation}
where    $Q_X(\hat{x},\hat{y})=x\cdot y=<\hat{x}\cup\hat{y},[X]>$ for $x,y\in H_2(X)$. The first Chern class, $c$, of a $\spinc$ structure is a \emph{characteristic class}, meaning that
\begin{equation}\label{eq:char}
\kappa(c)(x)\equiv x\cdot x \pmod 2
\end{equation}
for all $x\in H_2(X)$. For $\xi\in H^2(X,Y))$ we abbreviate $\xi^2\cong Q_X(\xi,\xi)$. In the case that $Y$ is a rational homology sphere this abbreviation can be extended as follows: for any $c\in H^2((X)$ we may choose $\xi\in H^2(X,Y)$ such that $\pi^*(\xi)=mc$ for some positive integer $m$. Then we define
\begin{equation}\label{eq:squareddef}
c^2\equiv \frac{Q_X(\xi,\xi)}{m^2}
\end{equation}
which can be shown to be independent of $\xi$. This is the meaning of the term $c_1(\t)^2$ in equations~\ref{eq:OSinequality} and ~\ref{eq:OSinequality2}.

\begin{proof}[Proof of Theorem~\ref{thm:onepositivedinvariants}]  Since $K\in \cP_1$, $K$ bounds a slice disk $\Delta$ in a $1$-positon, that is a $4$-manifold $V$ that satisfies the conditions of Definition~\ref{def:positive}.
 Let $X$ denote the $p^r$-fold cyclic cover of $V$ branched over $\Delta$, so  $\partial X=Y$. This makes sense according to the first paragraph in the proof of Proposition~\ref{prop:easy2fold}.  
 
First we will establish that the intersection form on $X$ is positive definite and diagonalizable.  Since $V$ is a $1$-positon for $K$, there exists a collection $S_i$ of surfaces disjointly embedded in $W=V\setminus\Delta$, representing a basis for $H_2(V)$, such that $\pi_1(S_i)\subset \pi_1(W)^{(1)}$. Thus this collection lifts to a collection $\widetilde{S}_i$ of $p^r\beta_2(V)$ embedded surfaces  in $X$. Indeed, their regular neighborhoods lift, so each has self-intersection $+1$. Moreover, for different $i$ and $j$ the collections of lifts $\widetilde{S}_i$ and $\widetilde{S}_j$ are disjoint, since $S_i$ and $S_j$ are disjoint in $V$. Each of these  surfaces represents a homology class $x_k$ such that $x_k\cdot x_k=+1$. The set of all such classes is linearly independent in $H_2(X)/TH_2(X)$, since if
$$
\sum a_kx_k=\mathrm{torsion}
$$
then taking the intersection of both sides with $x_k$ implies $a_k=0$.
Thus the surfaces $\widetilde{S}_i$ (and their translates) form a basis for a  free \emph{subgroup} (of rank $p^r\beta_2(V)$) of $H_2(X)/TH_2(X)$. We claim that it is a  \emph{direct summand}. For suppose some primitive element, $x$, of the lattice spanned by the $x_k$ were \emph{non}-primitive in $H_2(X)$/torsion, that is,
$$
x=\sum a_kx_k=my
$$
for some $m>1$. Then intersecting with $x_k$ implies that $m$ divides each $a_k$, contradicting the fact that $x$ was primitive. Finally we claim that $\widetilde{S}_i$ is a basis for $H_2(X)/TH_2(X)$, for which it now suffices to show that  $\beta_2(X)=p^r\beta_2(V)$.  But by Corollary~\ref{cor:ponealgslice}, $K$ is an algebraically slice knot so has zero signature function. Thus Proposition~\ref{prop:easy2fold} may be applied. In that proof it was shown that $\beta_2(X)=p^r\beta_2(V)$ (see equation~\ref{eq:rankH2}). Hence we have shown that there is a basis for which the intersection matrix for $X$ is an identity matrix. In particular the intersection form on $X$ is unimodular.

\begin{lem}\label{lem:onepositivecohomology} Suppose $X$ is a compact, oriented $4$-manifold whose intersection form is unimodular and whose boundary, $Y$, is a union of rational homology spheres.  Then the following two (isomorphic)  sequences are exact:
\begin{equation}\label{diag:torsionexact}
\begin{diagram}
\node{TH^2(X)} \arrow{e,t}{j^\ast}\arrow{s,r}{P.D.}\node{H^2(Y)}\arrow{s,r}{P.D.} \arrow{e,t}{\partial^\ast}
 \node{H^3(X,Y)}\arrow{s,r}{P.D.}
\\
\node{TH_2(X,Y)} \arrow{e,t}{\partial_*} \node{H_1(Y)} \arrow{e,t} {j_*} \node{H_1(X)}
\end{diagram}
\end{equation}
Moreover the kernel of $j_*$ is a metabolizer for the linking form on $H_1(Y)$.
\end{lem}

\begin{proof}  For the first claim it suffices to show the bottom sequence is exact. For this it suffices to show that any element in the kernel of $j_*$ is in the image of some torsion element. We have the following commutative diagram where $\lambda$ is the intersection form:
\[
\begin{diagram}
	\node{H_2(X)} \arrow{e,t}{\pi_*}\arrow{se,b}{\lambda}
   \node{H_2(X,Y)} \arrow{s,r}{\kappa\circ P.D.} \arrow{e,t}{\partial_*}
	\node{H_1(Y)}
\\
\node[2]{(H_2(X))^\ast}
\end{diagram}
\]
Since  the intersection form of $X$ is unimodular, $\lambda$ is surjective. Suppose $p\in\ker j_*$ and choose $x\in H_2(X,Y)$ so that $\partial_*(x)=p$. Take $y\in H_2(X)$ so that $\lambda(y)=(\kappa\circ P.D.)(x)$.  Then $\partial_*(x-\pi_*(y))=p$. Moreover $x-\pi_*(y)$ is in the kernel of $\kappa\circ P.D.$  and  hence is torsion.

Now set $G=\ker j_*:H_1(Y)\to H_1(X)$. We claim that $G$ is a metabolizer for the $\Q/\Z$ linking form on $H_1(Y)$. This is a  standard result in the case that $X$ is a rational homology ball, which is the case that $H_2(X)$ is torsion. But in fact the proof works as long as the sequence(s)  in Lemma~\ref{diag:torsionexact} is exact. We will provide some details. Consider the commutative diagram below, where $^\#$ denotes $Hom(-,\Q/\Z)$.

\[
\begin{diagram}
	\node{TH_2(X,Y)} \arrow{e,t}{\partial_*}\arrow{s,r}{PD}
   \node{H_1(Y)} \arrow{s,r}{PD} \arrow{e,t}{j_*}\node{TH_1(X)}
\\
\node{TH^2(X)}\arrow{e}\node{H^2(Y)}
\\
\node{H^1(X;\Q/\Z)}\arrow{n,r}{\beta_X}\arrow{s,r}{\kappa_X}\arrow{e}\node{H^1(Y;\Q/\Z)}\arrow{n,r}{\beta}\arrow{s,r}{\kappa}
\\
\node{(TH_1(X))^\#}\arrow{e,t}{j^\#}\node{(H_1(Y))^\#}
\end{diagram}
\]
The maps labelled $\beta$ and $\beta_X$ are Bockstein maps (connecting maps) in the long exact sequences associated to the coefficient sequence $0\to\Z\to\Q\to\Q/\Z\to 0$. Under our hypotheses one checks that $\beta$ is an isomorphism and $\beta_X$ is surjective. The right-most vertical composition is the definition of the linking form, sending $x$ to $\ell(x,-)$. The left-most vertical composition defines a ``linking pairing'' for $X$ (since the ambiguity due to the kernel of $\beta_X$ is lost when evaluated on $TH_1(X)$). The Kronecker map $\kappa$ is an isomorphism while $\kappa_X$ is surjective. Now, if $x, y\in G$, then $x=\partial_*(z)$ for some $z\in TH_2(X,Y)$. It follows that the homomorphism
$\ell(x,-)$ lies in the image of $j^\#$ , hence factors through $j_*$ and so annihilates $y$. This shows that $G\subset G^\perp$. Now suppose $x\in G^\perp$. Then the map $\ell(x,-)$ annihilates $G$ and hence descends to define an element of $(H_1(Y)/G)^\#$. But the latter embeds in $TH_1(X)$ and hence, since $\Q/\Z$ is an injective $\Z$-module, this map extends to an element of $(TH_1(X))^\#$. Thus $\ell(x,-)$ is in the image of $j^\#$. It follows that $x$ is in the image of $\partial_*$, so $x\in G$. Thus $G^\perp\subset G$.
	
\end{proof}

We now return to the proof of Theorem~\ref{thm:onepositivedinvariants}. Suppose  that $\s_0=\t_0|Y$ for some $\t_0\in\spinc(X)$. If $z\in G$, then the Poincar\'{e} dual satisfies $\hat{z}=j^*(\a)$ for some $\a\in H^2(X)$. To see that $\s_0+\hat{z}\in \spinc(Y)$ also extends to $X$, note that
$$
(\t_0+\a)|_Y=\t_0|_Y+j^\ast(\a)=\s_0+\hat{z}.
$$
Thus if $\s=\s_0+\hat{z}$ then $\s=\t|Y$ for some $\t\in\spinc(X)$.  By Lemma~\ref{lem:smallestchar} below, we may choose $\t$ such that $c_1(\t)^2=\beta_2(X)$.  Finally by applying the inequality ~\ref{eq:OSinequality2}, $d(Y,\s)\leq 0$.  

This concludes  the proof of Theorem~\ref{thm:onepositivedinvariants}, modulo the proof of Lemma~\ref{lem:smallestchar}.

 \begin{lem}\label{lem:smallestchar}   Suppose $X$ is a compact, oriented $4$-manifold whose intersection form is unimodular and whose boundary, $Y$, is a union of rational homology spheres.   Then for any $\s\in\spinc(Y)$ which is the restriction of some $\spinc$ structure on $X$, there is  $\t\in\spinc(X)$ with $\t|_Y=\s$ and $|c_1(\t)^2|\leq\beta_2(X)$  (with equality achievable precisely when the intersection form is isomorphic to $\pm\oplus <1>$).
\end{lem}
\begin{proof} By ~\cite[Theorem 1]{Elkies}, there exists a characteristic class $x\in H^2(X)$ for which $|x^2|\leq \beta_2(X)$. In the case(s) that the form is $\pm\oplus <1>$  with respect to the basis $\{ e_i \}$, then $x=e_1+\cdots +e_n$ is characteristic with $|x^2|=n=\beta_2(X)$. By ~\cite{Elkies} equality is achievable only in these cases.

Choose $\mathfrak{t}'\in \spinc (X)$ so that $c_1(\mathfrak{t}')=x$. Let $\s'=\t'|Y$. By hypothesis $\s=\t_0|_Y$ for some $\mathfrak{t}_0\in \spinc (X)$. Thus $\t_0=\t' + \alpha$ for some $\alpha\in H^2(X)$. Hence $\s=\s'+j^*(\alpha)$. But since $j*(\alpha)\in \ker \partial^*$, by Lemma~\ref{lem:onepositivecohomology},  $j^*(\alpha)=j^*(\beta)$ for some torsion class $\beta\in H^2(X)$. Let $\t\in \spinc$ be chosen such that $\t=\t'+ \beta$. Then
$$
\t|_Y=\t'|_Y+j^*(\beta)=\s'+j^*(\alpha)=\s
$$
as desired. Moreover, since $\beta$ is torsion, $c_1(\t)^2=c_1(\t')^2=x^2$.
 
 \end{proof}

\end{proof}

Of course, by reversing the orientation in Theorem~\ref{thm:onepositivedinvariants}, we have the analogous result for knots in $\cN_1$.

\begin{thm}\label{thm:onenegativedinvariants} If $K\in\cN_1$  and $Y$ is the $p^r$-fold cyclic  cover of $S^3$ branched over $K$, then there is a metabolizer $G'<H_1(Y)$ for the $\Q/\Z$-linking form on $H_1(Y)$; and there is a $\spinc$ structure $\mathfrak{s}_0$ on $Y$  such that $d(Y,\mathfrak{s}_0+\hat{z})\geq 0$ for all $z\in G'$, where $\hat{z}$ is the Poincare dual of $z$. If $X'$ is the $p^r$-fold branched cyclic  cover of the $1$-negaton $V'$  for $K$, then $G'$ may be taken to be the kernel of
$$
j_*:H_1(Y)\to H_1(X').
$$
Furthermore we may take $\mathfrak{s}_0$ to be a $\spinc$ structure corresponding to a spin structure on $Y$.  Furthermore, these correction terms will be even integers.
\end{thm}

\begin{proof} Suppose that $K\in \cN_1$ where $V'$ is a $1$-negaton for $K$. Then $-K\in\cP_1$ and $V=-V'$ where $V'$ is a $1$-positon for $-K$. Then Theorem~\ref{thm:onepositivedinvariants} can be applied to $-K$ to get the claimed result using the fact that $-d(Y,\mathfrak{s})=d(-Y,\mathfrak{s})$.
\end{proof}

\begin{cor}\label{cor:zerodeltainvariants} 	Suppose $K\in\cB_1$, $Y$ is the $p^r$-fold branched cyclic cover of $K$ and $\mathfrak{s}_0$ is a $\spinc$ structure corresponding to a spin structure on $Y$.  Then $d(Y,\mathfrak{s}_0)=0$. 
\end{cor}
\begin{proof} Take $z=0$ in both Theorem~\ref{thm:onepositivedinvariants} and Theorem~\ref{thm:onenegativedinvariants}.
\end{proof}

If $K$ is an oriented knot and $p/q\in \Q$ is non-zero then $S^3_{p/q}(K)$, the $p/q$-framed Dehn surgery, is a rational homology $3$-sphere whose first homology may be identified with $\Z_p$ via a canonical map sending the meridian to $1$. The $\spinc$-structures on $S^3_{p/q}(K)$ can be canonically labelled by elements $i\in \Z_p$ ~\cite[Section 2]{RubStrle}. 

\begin{cor}\label{cor:dinvariantssurgeries} 	Suppose $J\geq K$. Then for any non-zero $p/q\in \Q$,
$$
d(S^3_{p/q}(K),i)\geq d(S^3_{p/q}(J),i),
$$
In particular if $K\in\cB_0$, $Y$ is either the $+1$ or $-1$ framed surgery on $K$ and $\mathfrak{s}_0$ is the unique $\spinc$ structure on $Y$.  Then $d(Y,\mathfrak{s}_0)=0$. More generally, if $K\in \cB_0$ then the d-invariants of $S^3_{p/q}(K)$ agree with those of the lens space obtained by $p/q$-surgery on the unknot.
\end{cor}
\begin{proof}  Suppose $J\geq K$. By Proposition~\ref{prop:1surgery}, $Y=S^3_{p/q}(K)\coprod -S^3_{p/q}(J)$ bounds a $4$-manifold $X$ whose intersection form is isomorphic to $\oplus\langle -1\rangle$. In particular this intersection form  is  unimodular. Moreover the inclusion maps $\partial_{\pm}Y\to Y$ induce the \textit{identity} map $\Z_p\equiv H_1(\partial_+Y)\to H_1(\partial_-Y)\equiv \Z_p$. By Lemma~\ref{lem:smallestchar}, there is  $\t\in\spinc(X)$ with $\t|_Y=(i,i)$ and $c_1(\t)^2=-\beta_2(X)$. Thus by the Ozsv\'ath-Szabo inequality ~(\ref{eq:OSinequality}), $d(Y,(i,i))\geq 0$. Hence $d(S^3_{p/q}(K),i)\geq d(S^3_{p/q}(J),i)$.

Similarly if $K\geq J$ then $d(S^3_{p/q}(K),i)\leq d(S^3_{p/q}(J),i)$. Thus if $K\in \cB_0$ then, $K\geq U$ and  $U\geq K$, so letting $J=U$ we get the last claim of the Corollary.
\end{proof}

Since $d$-invariants associated to prime-power cyclic branched covers obstruct  membership in $\mathcal{T}_1$, in principle it should be possible to produce an infinite linearly independent set in $\mathcal{T}_0/\mathcal{T}_1$. However, due to the paucity of calculations of $d$-invariants for prime-power branched covers of topologically slice knots, at the moment we are only able to show:

\begin{cor}\label{cor:HLRexamples} The rank of
$\mathcal{T}/\mathcal{T}_1$
is at least one.
\end{cor}
\begin{proof}  We will show that any one of topologically slice knots considered by Hedden-Livingston-Ruberman in ~\cite{HLR} has infinite order in this quotient group. Let $K$ be a knot  appearing in that paper associated to a prime $p$. Suppose  $nK\in \mathcal{T}_1$ for some $n\neq0$.  Let $Y$ be the $2$-fold branched cover over $K$ and let $\mathfrak{s}_0$ denote a $\spinc$-structure on $Y$ associated to the unique spin structure on $Y$. By Corollary~\ref{cor:zerodeltainvariants}, $nd(Y,\mathfrak{s}_0)=0$.  But in ~\cite{HLR} it is calculated that $d(Y,\mathfrak{s}_0)$ is strictly less than $0$.
\end{proof}

\begin{cor}\label{cor:ambiguousdinvariants} If $K\in\cB_1$ and the $\Q/\Z$-valued linking form on $Y$,  the $p^r$-fold branched cover of $K$ has precisely one metabolizer, then the correction terms, $d(Y,\mathfrak{s}_0+\hat{z})$, from Theorem~\ref{thm:onenegativedinvariants} vanish (if $\mathfrak{s}_0$ arises from a $\spin$ structure on $Y$).
\end{cor}
\begin{proof} Under this hypothesis, the subgroup $G$ from Theorem~\ref{thm:onenegativedinvariants} (where $G$ is taken to be the kernel of $j_*$) will necessarily coincide with the subgroup $G'$ as in Corollary~\ref{thm:onepositivedinvariants}.  Thus, the $\spinc$ structures corresponding to $G=G'$ will yield vanishing $d$-invariants.
\end{proof}

%\begin{cor}\label{cor:samemetabolizerforposneg}
%	Suppose $K\in\cN\cP_1$ and $Y$ is a $p^r$-fold branched cyclic cover of $S^3$ over $K$ and that %there is a subgroup $G<H^2(Y)$ with $G=\im TH^2(X)=\im TH^2(W)$, where $X$ is a manifold as in %Lemma~\ref{lem:onepositivecohomology} and $W$ is a negative-definite analogue of $X$.  If $|G|$ is %odd, then the $|G|$ $d$-invariants from Theorem~\ref{thm:onenegativedinvariants} vanish.
%\end{cor}
%\begin{proof}
%	There is a bijection between the $\spinc$ structures of $Y$ extending over $X$ and those which extend %over $W$.  Under the hypotheses, this bijection is the identity map.  For, if $|G|$ is odd, then $H^2(Y;%\Z/2\Z)=0$ so $Y$ has a unique $\spin$ structure $\mathfrak{t}_0$. According to %Theorem~\ref{thm:onenegativedinvariants}, $\mathfrak{t}_0$ extends as a $\spinc$ structure to $X$ and %to $W$.
%\end{proof}

\begin{ex} Let $K=K(25,2)$ denote the two-bridge knot in Figure~\ref{fig:twist}.  In the notation of Conway, $K$ corresponds to the rational tangle $25/2=12+\frac{1}{2}$. We will argue that $K\in\cB_0-\cB_1$.   It is obvious that $K$ can be unknotted by changing one negative crossing or by changing six positive crossings.  Thus $K\in\cB_0$. The double-branched cover of $K$ is the lens space $L(25,2)$.  Since $\Z_{25}$ has only one proper subgroup, there is only one subgroup that is a metabolizer for the linking form. Therefore, if $K\in \cB_1$, then Corollary~\ref{cor:ambiguousdinvariants} implies that $L(25,2)$ has at least $5$ vanishing $d$-invariants. But using Ozsv\'ath and Szab\'o's formula for the $d$-invariants of lens spaces~\cite[Proposition 4.8]{OzSz1} (and a Mathematica notebook generously provided by Stanislav Jabuka), we compute the $d$-invariants of $L(25,2)$:
	$$\left\{-\frac{72}{25},-\frac{72}{25},-2,-2,-\frac{48}{25},-\frac{48}{25},-\frac{32}{25},-\frac{32}{25},-\frac{28}{25},-\frac{28}{25},-\frac{18}{25},-\frac{18}{25},-\frac{12}{25},-\frac{12}{25},-\frac{8}{25},-\frac{8}{25},\right.$$
	$$\left.-\frac{2}{25},-\frac{2}{25},0,0,0,\frac{8}{25},\frac{8}{25},\frac{12}{25},\frac{12}{25}\right\}$$
	\begin{center}
		\begin{figure}[ht]
			\begin{picture}(422,44)(0,0)
				\includegraphics[scale=2]{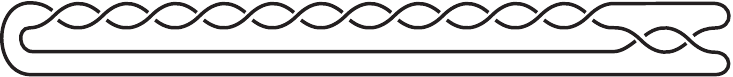}
			\end{picture}
			\caption{A twist knot}
			\label{fig:twist}
		\end{figure}
	\end{center}
There are only $3$ vanishing correction terms, so $K\notin\cB_1$.
\end{ex}

Sometimes $d$-invariants associated to prime-power branched obstruct membership in $\cP_0$, $\cN_0$,  and $\cB_0$, namely in cases where the branched cover is a homology sphere. This is the case (for any $p$) if $K$ is a knot with Alexander polynomial $1$. We state only the $\cB_0$ case.

\begin{cor}\label{cor:homologyspheredinvariants} Suppose $K\in\cB_0$ and $Y$, the $p^r$-fold branched cover over $K$, is an integral homology sphere, and  $\mathfrak{s}_0$ is the unique $\spinc$ structure on $Y$.  Then $d(Y,\mathfrak{s}_0)=0$.
\end{cor}
\begin{proof} Since $K\in \cB_0$, its $p^r$-signatures vanish by Proposition~\ref{prop:easy2fold}. Let  $X$ be the $p^r$-fold branched cover over the slice disk in a $0$-negaton. By Proposition~\ref{prop:easy2fold},  $X$ has negative definite intersection form and, since its boundary is a homology sphere, this form is unimodular. Hence by Lemma~\ref{lem:smallestchar}, there exists $\t\in\spinc(X)$ with $\t|_Y=\s_0$ and $-c_1(\t)^2\leq\beta_2(X)$.  Thus by the Ozsv\'ath-Szabo inequality ~(\ref{eq:OSinequality}), $d(Y,\mathfrak{s}_0)\geq 0$. Similarly using a $0$-positon we get $d(Y,\mathfrak{s}_0)\leq 0$. 
\end{proof}

\section{Nontriviality of the $\cB$ filtration}\label{sec:filtrationnotrivial}

In this section we prove that the successive quotients of the terms of the $\cB$ filtration have infinite rank and contain infinite rank $2$-torsion subgroups. Readers interested only in topologically slice knots can skip this section.

It is easy to create knots lying in $\cB_n$ using the satellite constructions discussed in Section~\ref{sec:examples}. In order to show that such a knot does not lie in $\cB_{n+1}$, it suffices, by Corollary~\ref{cor:positiveinratnsolv}, to show that it does not lie in $\mathcal{F}_{n.5}^\Q$ (which implies it doesn't lie in $\mathcal{F}_{n+1}^\Q$). For this, the techniques of ~\cite{CHL3,CHL5} using the von Neumann signature defects introduced in ~\cite{COT} can be successfully brought to bear.  Slight modifications are necessary because unfortunately the specific examples used in ~\cite{CHL3,CHL5} (in showing that $\mathcal{F}_n/\mathcal{F}_{n.5}$ has infinite rank) were not chosen to lie in $\cB_n$. 

The main idea can already be seen in the knot $K=R(LHT,RHT)$ of Example~\ref{ex:nine46}. There it was observed that, since the right-hand trefoil is a positive knot and the left-hand trefoil is a negative knot, $K\in \cB_1$. Yet $K$ cannot lie in $\cB_2$ since it would then lie in $\mathcal{F}_{1.5}^\Q$ and hence would have vanishing Casson-Gordon invariants (alternatively von Neumann metabelian signature invariants). Specifically the fact that the classical signatures of the right-handed or left-handed trefoil knot are \textit{non-zero} obstructs $K$ from being slice or even lying in $\cB_2$ (see Appendix~\ref{sec:CGinvts}).  It does not matter that the signs of these signatures are different.

\begin{thm}\label{thm:nontrivialityoffiltration} For each $n$, there is an infinite rank subgroup
$$
\Z^\infty\subset\frac{\cB_n}{\cB_{n+1}}.
$$
while for each $n\geq 0, n\neq 1$ there exists
$$
\Z_2^\infty\subset \frac{\cB_n}{\cB_{n+1}}.
$$
\end{thm}

Before proving the theorem, we note that, by happenstance the specific examples used in ~\cite[Theorem 3.3]{CDavis2} \textit{do} lie in $\cB_n$, and so \textit{that} theorem may be directly applied \textit{without modification} to show the slightly weaker result that for $n\geq 1$ there is an infinite rank subgroup
$$
\Z^\infty\subset\frac{\cB_n}{\cB_{n+2}}.
$$

\begin{proof}  First we treat the infinite rank claim. For $n=0$ we appeal to Theorem~\ref{thm:Endo}.   Now fix $n\geq 1$.  We follow the strategy of the proof of ~\cite[Theorem 8.1]{CHL3}. Let $\{J_0^j~|~1\leq j\leq \infty \}$ be the family of twist knots wherein $J_0^j$ has Seifert form $\bigl(\begin{smallmatrix}
-1&1\\ 0&-2j
\end{smallmatrix} \bigr)$. Each of these Arf invariant zero knots admits a positive projection and hence, by Proposition~\ref{prop:easyexs}, lies in $\mathcal{P}_0$. It was shown in ~\cite[Proposition 2.6]{COT2} that $\{\rho_0(J_0^j)\}$ is a linearly independent set of real numbers (here we view $\R$ is as a rational vector space). Let $R$ be the ribbon knot $9_{46}$. Let $\overline{T}^*$ denote either the left-hand trefoil or the connected sum of two copies of the left-handed trefoil (chosen so that $\rho^1(R)+\rho_0(\overline{T}^*)\neq 0$-see ~\cite[Theorem 8.1]{CHL3}). In either case, $\overline{T}^*$  lies in $\mathcal{N}_0$.  If necessary delete two members from the set $\{\rho_0(J_0^j)\}$ in order that \textit{afterwards} the set $\{\rho_0(J_0^j), \rho_0(\overline{T}^*), \rho^1(R)\}$ is a $\Q$-linearly independent set. Now define $J_1^j=R(\overline{T}^*,J_0^j)$, as shown on the right-hand side of Figure~\ref{fig:familyJ1}. Here the right-hand band of $R$ is tied into the knot $J^j_0$ and the left-hand band is tied into the knot $\overline{T}^*$. Let $U$ be the unknot. Since both $R(U,J_0^j)$ (shown on the left-hand side of Figure~\ref{fig:familyJ1}) and $R(\overline{T}^*,U)$ are also ribbon knots (in the first case cut the right-hand band, in the second case cut the other band),
\begin{figure}[htbp]
\setlength{\unitlength}{1pt}
\begin{picture}(327,151)
\put(0,0){\includegraphics{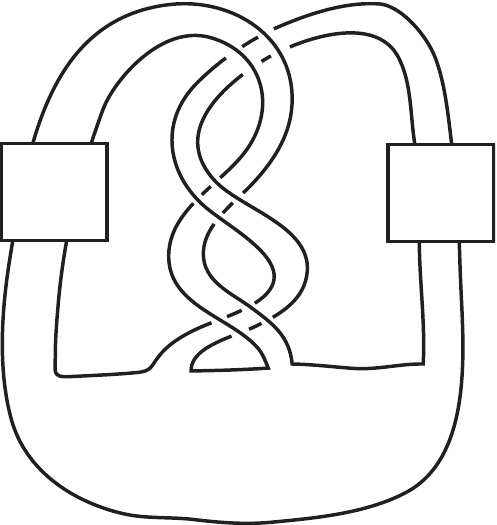}}
\put(202,0){\includegraphics{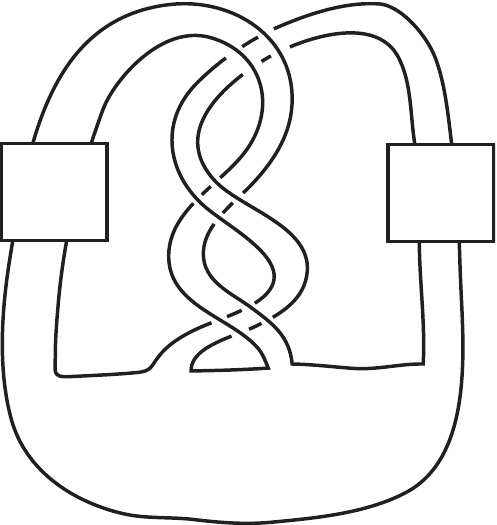}}
\put(140,20){$\equiv R(U,J_0^j)$}
\put(13,93){$U$}
\put(120,93){$J_0^j$}
\put(175,60){$J_1^j\equiv$}
\put(322,93){$J_0^j$}
\put(212,93){$\overline{T}^*$}
\end{picture}
\caption{$J_1^j=R(\overline{T}^*,J_0^j)$}\label{fig:familyJ1}
\end{figure}
$J_1^j$ may be viewed as the result of a doubling operator applied to $J_0^j$, namely
$$
J_1^j=R(\overline{T}^*,U)(J_0^j),
$$
or as the result of a doubling operator applied to $\overline{T}^*$. Hence, by Proposition~\ref{prop:operatorsact},
$$
 J_1^j\in \cB_1.
$$

\begin{figure}[ht!] \begin{center}
\begin{picture}(160,165)
\put(13,0){\includegraphics{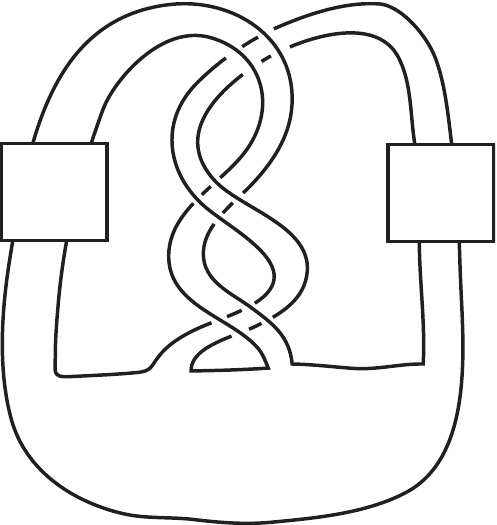}}
\put(25,90){$\overline{T}^*$}
\put(131,90){$J_{n-1}^j$}
\put(-15,72){$J^j_{n}=$}
\end{picture}
\caption{The family $J^j_{n}\in \cB_n$}\label{fig:familyJ_n^j}
\end{center} \end{figure}
Now for $n\geq 2$, define an infinite set of knots, $\{J_n^j~|~1\leq j\leq \infty \}$, inductively as shown in Figure~\ref{fig:familyJ_n^j}, by starting with
$J_1^j$. Since, for any fixed $n$ and $j$, $J_n^j$ is obtained from $J^j_1$ by applying a composition of $n-1$ doubling operators, we can apply Proposition~\ref{prop:operatorsact} to conclude that $J_n^j\in \cB_n$.

Suppose that some non-trivial linear combination of $\{J_n^j~|~1\leq j\leq \infty \}$ lay in $\cB_{n+1}$. Then, by Proposition~\ref{prop:positiveinratnsolv}, it would also lie in $\mathcal{F}_{n+1}^\Q$. But, in ~\cite[Proof of Theorem 8.1, Step 4,Step 2]{CHL3}, it was shown that this set is linearly independent modulo $\mathcal{F}_{n+1}^\Q$. Hence, this set is linearly independent in $\cB_n/\cB_{n+1}$.

For the $2$-torsion result for $n=0$, one simply considers the family of genus one negative amphichiral knots shown in ~\cite[Figure 2.5]{CHL6}. Since these can obviously be unknotted by changing either positive or negative crossings, each lies in $\cB_0$. If a non-trivial sum of these knots were to lie in $\cB_1$ then it would be algebraically slice by Corollary~\ref{cor:ponealgslice}. This is a contradiction since this family is well-known to be a set of  independent $2$-torsion elements in the algebraic knot concordance group.

For the $2$-torsion results  for $n>1$ the examples and results of ~\cite{CHL6} can be  applied with slight modification. First fix $n\geq 2$.  In ~\cite[Definition 2.6]{CHL6} families of knots $\mathcal{K}^n$ are defined by applying $n$ successive doubling operators operations to a seed knot $K_0$ where $K_0$ has very large classical signature. That is:
$$
\mathcal{K}^n\equiv \mathcal{R}^n\circ \dots \circ\mathcal{R}^2\circ\mathcal{R}^1(K_0)
$$
which ensures that $\mathcal{K}^n\in \mathcal{F}^{odd}_n$. But this does not ensure that $\mathcal{K}^n$ lies in $\cB_n$. To achieve this we replace $\mathcal{R}^1(K_0)$ in the above construction by $\mathcal{R}(P,N)$ so that now
$$
\mathcal{K}^n\equiv \mathcal{R}^n\circ \dots \circ\mathcal{R}^2\circ\mathcal{R}(P,N)
$$
where $R(P,N)$ as in Figure~\ref{fig:2-tor}.
\begin{figure}[ht!] \begin{center}
\begin{picture}(160,165)
\put(13,0){\includegraphics{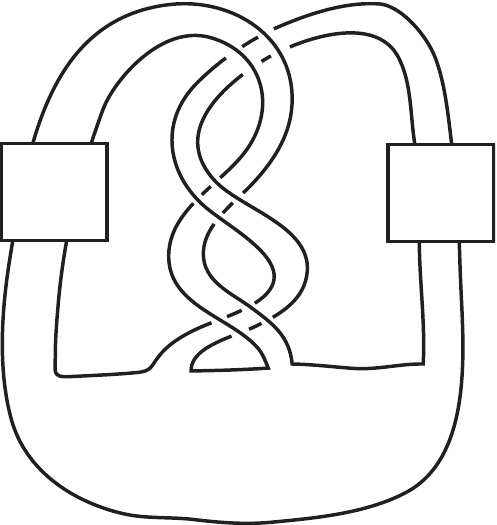}}
\put(23,92){$P$}
\put(133,92){$N$}
\end{picture}
\caption{$R(P,N)$}\label{fig:2-tor}
\end{center} \end{figure}
Choose $P$ to be a connected sum of a large number of right-handed trefoils and $N$ to be a connected sum of a large number of left-handed trefoils (the numbers specified below). Thus $P\in \cP_0$ and $N\in \cN_0$ so $\mathcal{R}(P,N)\in \cB_1$, just as in Example~\ref{ex:nine46}. It follows that $\mathcal{K}^n\in \cB_n$ by Proposition~\ref{prop:operatorsact}. Entire families are created by varying the Alexander polynomials of the ribbon knots $\mathcal{R}^1,\dots,\mathcal{R}^{n-1}$. Any knots thus created were shown to be negative amphichiral in ~\cite[Proposition 2.5]{CHL6} which implies that  these knots are of order at most $2$ in $\cB_n$. It remains to show that, for appropriate choice of $N$ and $P$, no sum of elements in this family lies in $\cB_{n+1}$.  Establishing this will show that these  knots are ``linearly independent'' elements of order two in $\mathcal{B}_n/\mathcal{B}_{n+1}$. By  Proposition~\ref{prop:positiveinratnsolv}, it suffices to show that no sum of elements in this family lies in $\mathcal{F}^\Q_{n+1}$. This was shown in ~\cite[Theorem 5.7]{CHL6}. However, in that paper, one begins with any  knot $K_0$ for which $\rho_0(K_0)$, the integral of the Levine-Tristram signature function, is very large. For our modified knot we need to specify a different condition. Let $C$ be twice the sum of the Cheeger-Gromov constants of the ribbon knots  $\mathcal{R}^1,\dots,\mathcal{R}^{n}$. This will depend on which ribbon knots are used and so will be different for each member of the family. Choose the number of left-handed trefoils so that $|\rho_0(N)|>C$. Then choose the number of right-handed trefoils so that 
$$
\rho_0(P)> C + |\rho_0(N)|.
$$
This will ensure that for each knot $\mathcal{K}^n$, each of the following is greater than $C$: $\rho_0(P)$, $\rho_0(N)$, $\rho_0(N)+\rho_0(P)+\rho^1(\mathcal{R})$. This is what is needed for the proof of ~\cite[Theorem 5.7]{CHL6}.
\end{proof}

\section{Topologically slice knots}\label{sec:mainthmknots}

The proof of the following theorem will essentially occupy the remainder of this paper, although at the end of this section we sketch the proof that, under an additional ``weak homotopy-ribbon'' assumption, the rank of $\frac{\T_n}{\T_{n+1}}$ is positive for each $n$.

\begin{thm}\label{thm:maintheoremtop}  The rank of $\frac{\T_1}{\T_2}$ is positive.
\end{thm}

\begin{proof} Let $R$ be the ribbon knot $9_{46}$, and $\eta_1. \eta_2$ be meridional circles to the left and right-hand bands of $R$, respectively, as shown on the left-hand side of Figure~\ref{fig:mainexnotT2}. Let $J$ be the twist knot with a negative clasp and $11$ twists and let $T$ be $Wh(RHT)$ where $RHT$ is the right-handed trefoil knot. Let $\eta_1, \eta_2$ also denote meridional circles to the left and right-hand bands of $K$, respectively. 
\begin{figure}[htbp]
\setlength{\unitlength}{1pt}
\begin{picture}(327,151)

\put(201,75){}
\put(-41,85){$R\equiv$}
\put(-32,58){$\eta_1$}
\put(130,58){$\eta_2$}
\put(190,0){\includegraphics{family_scaled}}

\put(203,93){$J$}
\put(165,70){$K\equiv$}
\put(313,93){$T$}
\put(-20,0){\includegraphics{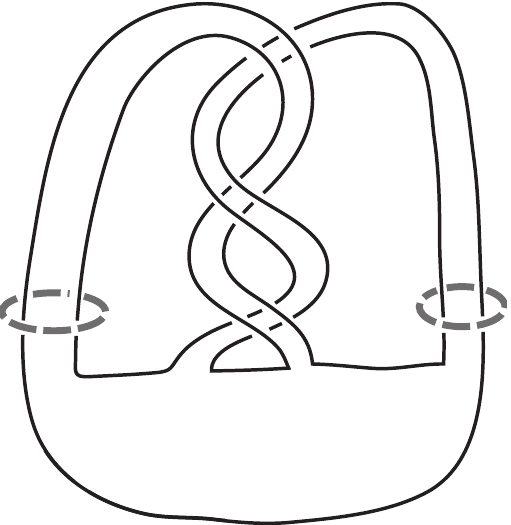}}
\end{picture}
\caption{$K=R(J,T)$}\label{fig:mainexnotT2}
\end{figure}

We will prove that $K$ is infinite order in $\T_1/\T_2$. First we show that $K\in \T_1$. Note that both $R(J,U)$ and $R(U,T)$, where $U$ is the unknot, are  ribbon knots, as we saw in Example~\ref{ex:nine46}.  Since $T=Wh(RHT)$ has Alexander polynomial $1$, it is a topologically slice knot by ~\cite{FQ}, and  lies in $\cP_1$ by  Example~\ref{ex:exswhitehead}. Hence $K$ is a satellite knot whose  pattern $R(J,U)$ is a ribbon knot and whose companion, $Wh(T)$, is a topologically slice knot that lies in $\cP_1$. It follows that $K$ is a topologically slice knot, that is $K\in \T$.  Moreover since $R(J,U)$ is slice, it lies in $\cP_2$. Thus, since $\eta_2$ lies in the commutator subgroup, $K\in \cP_2$ by equation~\ref{eq:operators} of Proposition~\ref{prop:operatorsact}. Since $J$ can be unknotted either via positive crossings or negative crossings, $J\in \cB_0$. Since $K$ may also be viewed as a winding number zero satellite knot with  ribbon knot pattern $R(U,T)$ and with companion $J$, $K\in \cB_1$ by equation~\ref{eq:operators3} of Proposition~\ref{prop:operatorsact}. Hence $K\in \T\cap \cB_1\equiv\T_1$.

In the rest of the proof  we show $K\notin \cN_2$ which shows $K$ is non-zero in $\T_1/\T_2$. It then will follow immediately from Corollary~\ref{cor:torsionfreeC}  that (since $K\in \cP_2$) no positive multiple of $K$ lies in $\cB_2$.  Consequently no non-zero  multiple of $K$ lies in $\T_2$, so $K$ has infinite order in $\T_1/\T_2$.

We proceed by contradiction. Suppose that $K\in \cN_2$ so $K$ bounds a disk $\Delta$  in some $V$ that satisfies (the negative definite analogue of)  Definition~\ref{def:positive}. Let $\Sigma$ denote the $3$-fold cyclic cover of $S^3$ branched over $K$. It is easy to see that $H_1(\Sigma)$ is independent of $J$ and $T$ and in fact $H_1(\Sigma)\cong \Z_7\langle x_1\rangle\oplus\Z_7\langle y_1\rangle$, and that there are only two metabolizers, $\langle x_1\rangle$ and $\langle y_1\rangle$,  for the $\Q/\Z$-valued linking form. A surgery picture of the $3$-fold cover of $S^3$ branched over $K$ (in the case that $J=U$) is shown in Figure~\ref{3fold}, using the method of ~\cite{AkKir}. The homology classes that we call $x_1$ and $y_1$ are represented by the \textit{meridians} of the oriented circles labelled by $x_1$ and $y_1$ in this figure.

As we saw in the proof of Proposition~\ref{prop:easy2fold} (even under the weaker hypothesis $K\in \cN_0$), the $3$-fold cyclic cover of $V$ branched over $\Delta$, here denoted  $X=\widetilde{V}$, is defined and has boundary $\Sigma$. Since $K\in \cN_1$,  and $V$ is a $1$-negaton, by Theorem~\ref{thm:onenegativedinvariants}  the kernel of
$$
j_*:H_1(\Sigma)\to H_1(X)
$$
is a metabolizer for the linking form. Hence the proof splits into two cases, $\ker(j_*)=<x_1>$ and $\ker(j_*)=<y_1>$.

\textbf{Case I}: $\ker(j_*)=<x_1>$

In this case we will show that even the weaker hypothesis $K\in \cN_1$ leads to a contradiction.

\begin{lem}\label{lem:ignoreJ} Without loss of generality we may assume that $J=U$. Specifically, if $K\in\cN_1$ via a $1$-negaton for which Case I holds then the same is true for the knot obtained by letting $J$ be the unknot.
\end{lem}
\begin{proof}  Since $J$ may be unknotted by changing only positive crossings, $J\geq_0 U$. It follows that $K=R(J,T)\geq _1 R(U,T)$.  Since $K\in \cN_1$, or equivalently, $U\geq_1 K$,  and since $K\geq_1 R(U,T)$, it follows that $U\geq_1 R(U,T)$ or, equivalently, $R(U,T)\in \cN_1$.  Temporarily let $K'=R(U,T)$, and let $\Sigma'$ and $X'$ denote the analogous branched covers. More precisely, $K\geq_1 K'$ means, after reversing orientation, that  $K'$ is concordant to $K$ inside a \textit{negative} definite $4$-manifold $C$ with intersection form $\oplus <-1>$. Thus $K'$ is slice in $V'\equiv V\cup C$.  Moreover, since the $J$ does not affect the Alexander module of $K$, $H_1(\Sigma)$ is naturally isomorphic to $H_1(\Sigma')$ and under this identification the kernel of
$$
j_*': H_1(\Sigma ')\to H_1(X')=H_1(\widetilde{V\cup C})
$$
is once again $<x_{1}'>$. This finishes the proof of the Lemma.
\end{proof}

Therefore if setting $J=U$ leads to a contradiction then we will be done with Case I. 
Henceforth, \textbf{in our proof of Case I}, we will just \textbf{assume that} $J=U$, so $K$ is redefined as $R(U,T)$ for the rest of Case I (i.e. we drop the ``primes'' ).

A surgery picture of the $3$-fold cover of $S^3$ branched over $K$ is shown in Figure~\ref{3fold}.

\begin{figure}[h!]
	\begin{center}
		\begin{picture}(346,165)(0,0)
			\put(80,125){$T$}
			\put(210,125){$T$}
			\put(335,125){$T$}
			
			\put(72,-1.8){$>$}
			\put(72,-8){$x_1$}
			
			\put(190,-1.8){$<$}
			\put(190,-8){$x_2$}
			
			\put(72,28.5){$<$}
			\put(72,20){$y_1$}
			
			\put(190,28.5){$>$}
			\put(190,20){$y_2$}
			
			\put(28,155){$0$}
			\put(58,155){$0$}
			\put(280,155){$0$}
			\put(315,155){$0$}
			
			\includegraphics[scale=1.5]{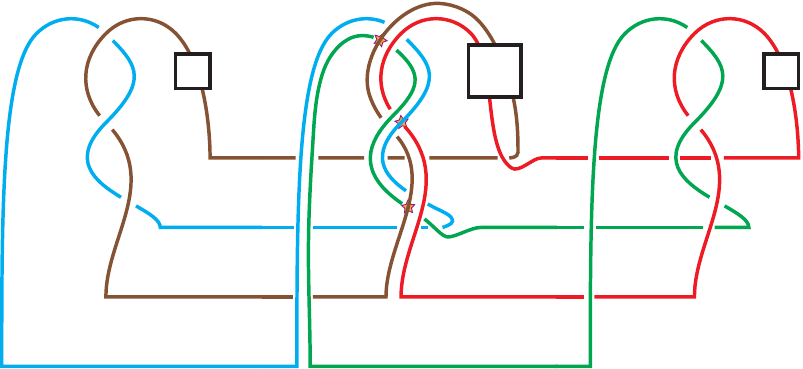}
		\end{picture}
		\caption{The $3$-fold branched cover $\Sigma_3$}
		\label{3fold}
	\end{center}
\end{figure}
Therefore, by Theorem~\ref{thm:onenegativedinvariants}, for any  $z$ in $<x_1>$, 
$$
d(\Sigma,\mathfrak{s}_0+\hat{z})\geq 0,
$$
where $\mathfrak{s}_0$ is the spin$^c$ structure associated to the unique spin structure on $\Sigma$.

But in Theorem~\ref{thm:dinvtcalc} of the first appendix, Section~\ref{sec:calcdinvt}, we  calculate that $d(\Sigma,\mathfrak{s}_0+4\widehat{x_2})\leq -3/2$, which is a contradiction since $4x_2=x_1$ in $H_1(\Sigma)$.   This completes the proof of Case I.

\vspace{.15in}

\textbf{Case II}: $\ker(j_*)=<y_1>$

\vspace{.15in}

We return to $K\equiv R(J,T)$. We will show, using Casson-Gordon invariants, that $K$ admits no such $2$-negaton \textit{even in the topological category}. This is counter-intuitive because, since $K$ is topologically slice, all the \textit{Casson-Gordon sliceness obstructions} must be zero!  The resolution of this apparent paradox is simple. Casson-Gordon \textit{invariants} involve a choice of a metabolizer. The so-called \textit{Casson-Gordon sliceness obstruction} is that \textit{one} (not all) of the Casson-Gordon \textit{invariants} must be zero. So, even for topologically slice knots, Casson-Gordon \textit{invariants} obstruct the knots being slice in such a way that some fixed metabolizer arises.  

By  Proposition~\ref{prop:positiveinratnsolv} since $K\in \cN_2$,  $K\in \mathcal{F}_2^{odd}$ and hence $K\in \mathcal{F}_{1.5}^{odd}$. More precisely, we are assuming that $K$ admits a $2$-negaton $V$ for which  Case II holds (a condition on the $3$-fold branched cover). Recall that in the proof  of Proposition~\ref{prop:positiveinratnsolv}, given a $2$-negaton, $V$, we showed $K\in \mathcal{F}_2^{odd}$ by connected summing with copies of $\mathbb{C}P(2)$. Since this does  not alter $\pi_1$, we know that $K\in \mathcal{F}_2^{odd}$ in such a way that it is slice via a particular $4$-manifold $V$ as in Definition~\ref{def:altdefoddsolv} ($n=2$) wherein the kernel of the corresponding map $j_*$ on $3$-fold branched covers is $<y_1>$.

In the second appendix, Section~\ref{sec:CGinvts}, we show that  Casson-Gordon invariants of $K$ obstruct membership in $\mathcal{F}_{1.5}^{odd}$ (something that was essentially already proved in ~\cite[Theorem 9.11]{COT}).  In the case at hand the relevant Casson-Gordon invariants of $K$ are equal to certain well-known algebraic concordance invariants of $J$ (essentially follows from previous work of Litherland, Gilmer and Livingston), which we compute to be non-zero (in Corollary~\ref{cor:ourknotCG}). This is all detailed in Appendix~\ref{sec:CGinvts}. Relying on the previous work alluded to above is delicate, since (due to an error in an early paper of Gilmer) many of the results stated in the relevant papers are stronger than can currently be proved (some corrected statements appear in ~\cite{GL4}).  Moreover we do not need the full strength of Gilmer's work anyway, so the gap in his proof is not relevant to us. But this makes things challenging  even for the  experts and so we have relegated this to an appendix.

This completes the proof of Theorem~\ref{thm:maintheoremtop} .

\end{proof}

At this time we are unable to show  that  $\mathcal{T}_n/\mathcal{T}_{n+1}$ is non-zero for every $n$.  However we can exhibit with confidence precise families that will, hopefully, in the  future be shown to give infinite linearly independent sets in $\mathcal{T}_n/\mathcal{T}_{n+1}$. These are given below. Moreover we do prove below, under a mild ``homotopy-ribbon assumption'', that $\mathcal{T}_n/\mathcal{T}_{n+1}$ has positive rank for every $n>1$. We view this assumption as a technical hypothesis. We could also avoid this hypothesis and prove the anticipated result if we were better at calculating $d$-invariants of \textit{all} branched covers of a fixed knot.

\begin{thm}\label{thm:generalnontriviality} For each $n>1$ there is a (topologically slice) knot, $K_n$, in $\mathcal{T}_n$ for which there exists no $(n+1)$-negaton $V$ wherein the inclusion  $S^3-K\to V\setminus \Delta$ induces a  surjection on Alexander modules.  
\end{thm}

\begin{proof} Because this theorem is not the desired end-goal, we only sketch the proof. In fact the examples and the proof are entirely similar to those of the previous theorem, but with  the role of Casson-Gordon invariants being replaced by higher-order signatures. Again let $R$ be the ribbon knot $9_{46}$ as shown on the left-hand side of Figure~\ref{fig:mainexnotT2}  and let $T$ be $Wh(RHT)$.  We will specify below a knot $J_{n-1}\in\cB_{n-1}$ and then we set $K=K_n:= R(J_{n-1},T)$ as shown in Figure~\ref{fig:familyJn}.
\begin{figure}[htbp]
\setlength{\unitlength}{1pt}
\begin{picture}(327,151)

\put(201,75){}

\put(100,0){\includegraphics{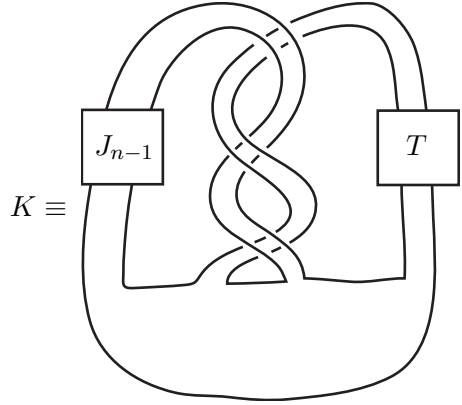}}

\put(105,93){$J_{n-1}$}
\put(73,70){$K\equiv$}
\put(223,93){$T$}

\end{picture}
\caption{$K_n=R(J_{n-1},T)$}\label{fig:familyJn}
\end{figure}
Thus each knot $K_n$ is very similar to the example of the previous theorem as shown on the right-hand side  of Figure~\ref{fig:mainexnotT2}. As in the previous proof, since $R$ is a ribbon knot and $T$ is topologically slice, $K_n$ is topologically slice. Also, as in the previous proof, since $R(U,T)$ is a ribbon knot and $J_{n-1}\in \cB_{n-1}$, $K_n\in \cB_n$. Hence $K_n\in \mathcal{T}_n$. 

We now  recursively specify $J_{n-1}$.  Let $J_1$ be  as in Figure~\ref{fig:J1} where  $J_0^+$ is the right-handed trefoil knot and  $J_0^-$ is the connected-sum of two left-handed trefoil knots. 
\begin{figure}[ht!] \begin{center}
\begin{picture}(160,165)
\put(13,0){\includegraphics{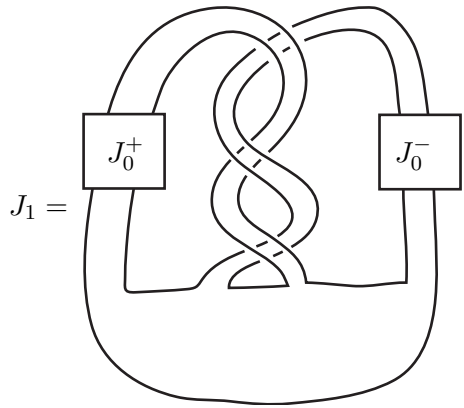}}
\put(22,92){$J_0^+$}
\put(131,92){$J_{0}^-$}
\put(-15,72){$J_{1}=$}
\end{picture}
\caption{The knot $J_1$}\label{fig:J1}
\end{center} \end{figure}
For $n>2$ let $J_{n-1}$ be given inductively as on the right-hand side of Figure~\ref{fig:familyJn-1}.  
\begin{figure}[ht]
\setlength{\unitlength}{1pt}
\begin{picture}(327,151)
%\put(-20,0){\includegraphics{ribbon_family}}
\put(74,0){\includegraphics{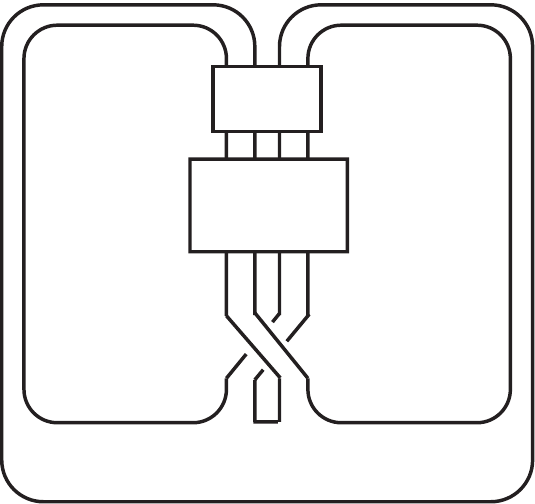}}
%\put(48,113){$-2$}
\put(142,83){$J_{n-2}$}
\put(144,113){$-2$}
\put(28,56){$J_{n-1}\equiv$}
%\put(-50,56){$R^{1}\equiv$}
\end{picture}
\caption{$J_{n-1},~ n\geq 3$}\label{fig:familyJn-1}
\end{figure}
We owe the reader a proof that $J_{n-1}\in \cB_{n-1}$.  Note in Figure~\ref{fig:familyJn-1} that the knot obtained by replacing $J_{n-2}$ by an unknot is itself a ribbon knot $R'$. Thus $J_{n-2}=R'(J_{n-2})$. Hence, by Corollary~\ref{cor:satellites}, it suffices to show  inductively that $J_{n-2}\in \cB_{n-2}$. The base step of the induction is the assertion that $J_1$ of Figure~\ref{fig:J1} lies in $\cB_1$. But this was already shown in Example~\ref{ex:nine46}, since $J_0^+\in \cP_0$ and $J_0^-\in \cN_0$. (The additional feature of $J_{n-1}$ that will implicitly be used later is that it does not lie in $\cB_n\subset\mathcal{F}^{odd}_n$ as detected by higher-order von Neumann signature defects  ~\cite[Theorem 9.1]{CHL3}.)

We proceed by contradiction.  Suppose that $K=K_n\in \cN_{n+1}$ so $K$ bounds a disk $\Delta$  in some $(n+1)$-negaton $V$ that satisfies (the negative definite analogue of)  Definition~\ref{def:positive} and wherein the inclusion  $S^3-K\to V\setminus\Delta$ induces a  surjection on Alexander modules. Let $\Sigma$ denote the $3$-fold cyclic cover of $S^3$ branched over $K$ and let $X$ be the $3$-fold cyclic cover of $V$ branched over $\Delta$.  Let $\eta_1, \eta_2$  denote meridional circles to the left and right-hand bands of $K$, respectively. As we saw in the previous proof,  the kernel of
$$
j_*:H_1(\Sigma)\to H_1(X)
$$
is a metabolizer for the linking form. Moreover,  $H_1(\Sigma)\cong \Z_7\langle x_1\rangle\oplus\Z_7\langle y_1\rangle$, and  there are only two metabolizers, $\langle x_1\rangle$ and $\langle y_1\rangle$,  for the $\Q/\Z$-valued linking form, where $x_1$ is a lift of $\eta_1$ and $y_1$ is a lift of $\eta_2$. Hence the proof splits into two cases, $\ker(j_*)=<x_1>=<\tilde{\eta}_1>$ and $\ker(j_*)=<y_1>=<\tilde{\eta}_2>$.

\textbf{Case I}: $\ker(j_*)=<x_1>$

In this case we will show that even the weaker hypothesis $K\in \cN_1$ leads to a contradiction. In fact the proof is identical to the proof of Case I of Theorem~\ref{thm:maintheoremtop}, since in Lemma~\ref{lem:ignoreJ}, we showed that we may ignore $J_{n-1}$, in which case our knot $K$ is identical to that of Theorem~\ref{thm:maintheoremtop}! Moreover here we do not need to use our weak homotopy ribbon assumption.

\textbf{Case II}: $\ker(j_*)=<y_1>=<\tilde{\eta}_2>$ 

Since in this case our proof will work in the topological category and since $T$ is topologically slice, it may be ignored. The first step in the proof of this case is to use the weak homotopy ribbon assumption to show Case II implies what we call Case II$^\prime$ which is defined as follows. The inclusion map $j:S^3\setminus K\to V\setminus\Delta$ induces a map $\tilde{j}_*$ of integral Alexander modules and the kernel of the composition
$$
\pi\circ\tilde{j}_*:\mathcal{A}(K)\to\mathcal{A}(V\setminus\Delta)\to \mathcal{A}(V\setminus\Delta)\otimes \Q
$$
 is known to be self-annihilating with respect to the classical Blanchfield linking form as long as $n\geq 0$  ~\cite[Theorem 4.4]{COT}. In our case, $\mathcal{A}(K)$ has precisely two such submodules, generated by $\eta_1$ and $\eta_2$ respectively. Then we define:

\textbf{Case II$^\prime$}: The kernel of $\pi\circ\tilde{j}_*$ is the submodule generated by $\eta_2$.

\noindent We show Case II implies Case II$^\prime$. This is the only place where we need our hypothesis that $\tilde{j}_*$ is surjective.
For this implies that $\mathcal{A}(V\setminus\Delta)$ is $\Z$-torsion-free (like $\mathcal{A}(K)$) and so $\pi$ is injective. Suppose that Case II$^\prime$ were to fail to hold. Then $\pi\circ\tilde{j}_*(\eta_1)=0$ so $\tilde{j}_*(\eta_1)=0$. Since
$$
\mathcal{A}(V\setminus\Delta)\cong \frac{G^{(1)}}{G^{(2)}}
$$
where $G=\pi_1(V\setminus\Delta)$, it would follow that, as a \textit{homotopy} class, $j_*(\eta_1)\in G^{(2)}$. It follows that, as a homotopy class, $j_*(\tilde{\eta_1})\in \pi_1(X)^{(1)}$. Thus as a \textit{homology} class $j_*(\tilde{\eta}_1)=0$, which contradicts our assumption that we are in Case II.

Assuming now that we are in Case II$^\prime$, we can prove that $K\notin \cN_{n+1} $ by showing that it is not in $\mathcal{F}_{n.5}^{\Q}$.   This proof is essentially identical to the proofs of ~\cite[Theorem 9.1, Thm 8.1, Ex.8.4]{CHL3}.  The only difference is that in that proof \textit{both} $J_0^+$ and $J_0^-$ were chosen to have classical  signatures of the same sign. Here, in order to get $J_1\in \cB_1$ we needed to choose $J_0^-\in \cN_0$ and $J_0^+\in \cP_0$ so they will have signatures of opposite signs. However, as long as the \textit{sum} of the integrals of their Levine-Tristram signature functions is large than a certain constant $\rho^1(9_{46})$ (which is now known due to work of Christopher Davis to be bounded in absolute value by $1$), the proof works the same.

This concludes our sketch of the proof of Theorem~\ref{thm:generalnontriviality}. The only use made of the extra hypothesis that $j_*$ is surjective is to ensure that Case II and Case II$^\prime$  coincide. This could be avoided by doing Case I for branched covers for arbitrarily large primes (not just $3$-fold).
\end{proof}

If the extra assumption could be eliminated, then in order the get subgroups of infinite rank, one would merely replace the $9_{46}$ knot $R$ by the family in Figure~\ref{fig:differentR}. The concordance classes of the resulting families knots $K^m_n=R^m(J_{n-1},T)$ could be distinguished (fixed $n$, varying $m$) by virtue of their coprime Alexander polynomials  using, in Case I, the techniques of ~\cite{HLR} (of $d$-invariants associated to different $p$-fold branched covers); and, in Case II, the techniques of ~\cite{CHL5}.
\begin{figure}[htbp]
\setlength{\unitlength}{1pt}
\begin{picture}(327,151)
\put(90,0){\includegraphics{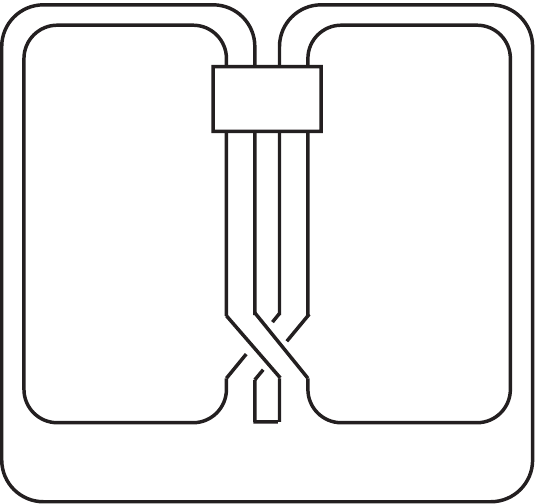}}
`\put(160,113){$m$}
\end{picture}
\caption{Family of ribbon knots, $R^m$}\label{fig:differentR}
\end{figure}

\section{Appendix I: Calculation of a certain $d$-invariant}\label{sec:calcdinvt}

We employ the notation of the proof of Theorem~\ref{thm:maintheoremtop} Case 1 (where we have already assumed that $J=U$). Recall that the $3$-fold branched cover, $\Sigma$, of $K\equiv R(U,T)$ is pictured in Figure~\ref{3foldappendix}. In this appendix, we abuse notation in that we will use the symbols $x_1,x_2,y_1,y_2$ to denote both the oriented circles shown in Figure~\ref{3foldappendix}, as well as the first homology classes represented by the oriented meridians of these circles.
\begin{figure}[!ht]
	\begin{center}
		\begin{picture}(346,165)(0,0)
			\put(80,125){$T$}
			\put(210,125){$T$}
			\put(335,125){$T$}
			
			\put(72,-1.8){$>$}
			\put(72,-8){$x_1$}
			
			\put(190,-1.8){$<$}
			\put(190,-8){$x_2$}
			
			\put(72,28.5){$<$}
			\put(72,20){$y_1$}
			
			\put(190,28.5){$>$}
			\put(190,20){$y_2$}
			
			\put(28,155){$0$}
			\put(58,155){$0$}
			\put(280,155){$0$}
			\put(315,155){$0$}
			
			\includegraphics[scale=1.5]{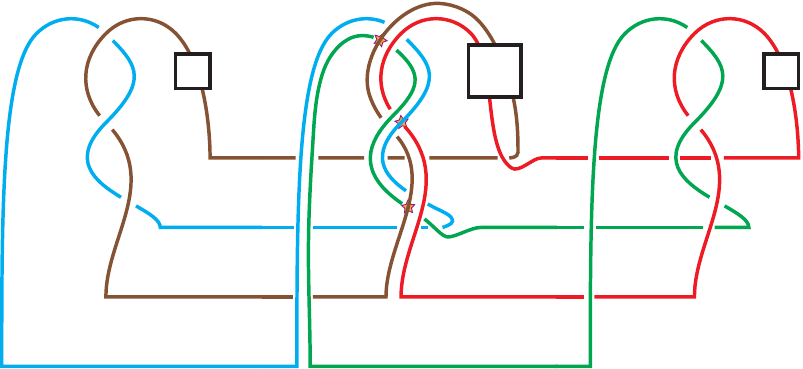}
		\end{picture}
		\caption{The $3$-fold branched cover $\Sigma$}
		\label{3foldappendix}
	\end{center}
\end{figure}
Our goal is to prove:
\begin{thm}\label{thm:dinvtcalc}
	$d(\Sigma,\sss_0+4\widehat{x_2})\leq -3/2$.
\end{thm}

\begin{proof} Momentarily ignoring the knots $T$, observe that that changing the $3$ crossings with stars on them would enable one to pull apart $\{\textrm{the curves labeled $x_1$ and $y_1$}\}$ from $\{\textrm{the curves labeled $x_2$ and $y_2$}\}$.  Recall also that one can achieve a crossing change by adding a $-1$-framed curve and blowing it down, as shown in Figure~\ref{blowdown}. We refer the reader to \cite{GompfStip} for the specifics of Kirby calculus.

\begin{figure}[h!]
	\begin{center}
		\begin{picture}(222,84)(0,0)
			\put(80,50){$-1$}
			\put(103,40){\Huge$\leadsto$}
			\put(93,30){blow down}
			\includegraphics[scale=2]{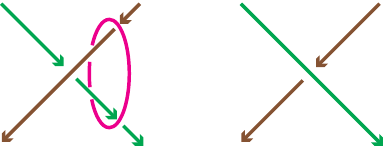}
		\end{picture}
		\caption{Blowing down a $-1$}
		\label{blowdown}
	\end{center}
\end{figure}
Corresponding to the three indicated crossing-changes, we now describe a $4$-dimensional cobordism $W$, one of whose boundary components is $-\Sigma$ and the other we call $Y$. Begin with $\Sigma\times[0,1]$ and form $W$ by attaching to $\Sigma\times \{1\}$ four $-1$ framed $2$-handles along curves $\{e_1,e_2,e_3,e_4\}$  where $e_4$ is a meridian of $x_2$ and $e_1,e_2,e_3$ are the above-mentioned small circles about the starred crossings, like the $-1$-framed curve in Figure~\ref{blowdown}, forming the homology classes $e_1=y_1-x_2$, $e_2=y_1-x_2$, $e_3=x_1-y_2$.  $Y$ is the $3$-manifold obtained from framed surgery on the union of the framed link in Figure~\ref{3foldappendix} together with the four $-1$-framed circles.  To simplify $Y$,  we  blow down all four of the $-1$s.  Since each $-1$-framed curve has linking number one with the curves that pass through it, each blow-down increases the framings of these curves involved by $1$.  The resulting framed link description of $Y$ is shown in  Figure~\ref{topboundary}.

\begin{figure}[!ht]
	\begin{center}
		\begin{picture}(153,82)(0,0)
			\put(23,35){$3$}
			\put(123,35){$3$}
			\put(72,35){$T$}
			\put(53,64){$T$}
			\put(93,64){$T$}
			\put(10,75){$1$}
			\put(29,75){$2$}
			\put(120,75){$1$}
			\put(140,75){$3$}
			\includegraphics{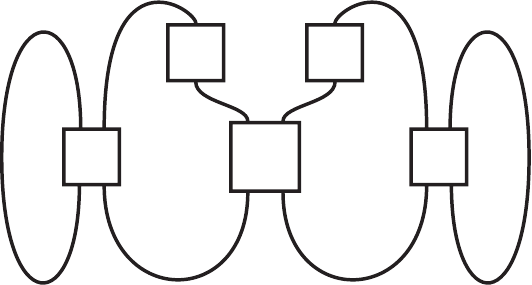}
		\end{picture}
		\caption{The top boundary $Y$}
		\label{topboundary}
	\end{center}
\end{figure}

A presentation matrix for $H_1(\Sigma)$ with respect to the basis given by the meridians $\{x_1,x_2,y_1,y_2\}$ is given by the framing matrix $P$ of the framed link in Figure~\ref{3foldappendix}.
$$
P=\left(\begin{array}{cccc}
0 & 0 & 3 & 1\\
0 & 0 & 2& 3\\
3 & 2& 0 & 0\\
1& 3 &0 &0
\end{array}\right).
$$
From this we see that $H_1(\Sigma)\cong \Z_7\oplus\Z_7$ generated by $\{x_1,y_1\}$ and where $x_2=4x_1$ and $y_2=4y_1$. Hence the order of each of $x_i,y_i,e_i$ is $7$. It also follows that $H_1(W)=0$.

We now compute the intersection form on $H_2(W)$. Since $\Sigma$ is a rational homology sphere $H_2(W)\cong \Z^4$ with basis $\{E_1,E_2,E_3,E_4\}$  given as follows.  Since $e_i$ has order $7$ in $H_1(\Sigma)$, there is a surface, $F_i$, in $\Sigma\times \{1\}$ with boundary $7\,e_i$.  Then $E_i$ is represented by the union of this surface and seven copies of the core of the $2$-handle  attached along $e_i$.  Note that \[ E_i\cdot E_j = E_i\cdot F_j = 7e_i\cdot F_j= \lk_{\Sigma}(7\,e_i,7\,e_j) = 49\,\lk_{\Sigma}(e_i,e_j)\]
Since the oriented curves $e_i$ miss the four-component framed link (in Figure~\ref{3foldappendix}), the linking number in $\Sigma$ is related to the linking number in $S^3$ by the formula \[\lk_{\Sigma}(e_i,e_j) = \lk_{S^3}(e_i,e_j) - [e_i]^\top\cdot P\inv \cdot [e_j]\] where $[\_]$ denotes the coordinates with respect to $\{x_1,x_2,y_1,y_2\}$ and where, if $i=j$, linking number is interpreted as framing. We refer the reader to \cite[Lemma 1.1]{Hoste} for a proof.  Since the matrix that expresses $\{e_1,\dots,e_4\}$ in terms of $\{x_1,x_2,y_1,y_2\}$ is
$$
Q=\left(\begin{array}{rrrr}
0 & 0 & 1 & 0\\
-1 & -1 & 0& 1\\
1 & 1& 0 & 0\\
0& 0 &-1 & 0
\end{array}\right),
$$ 
we compute that the matrix, $\mathcal{I}$, for the intersection form on $W$, with respect to the basis $\{E_1,E_2,E_3,E_4\}$  is 
 \[ \mathcal{I} =49 \lk_{\Sigma} =49\left( -I-Q^T P^{-1}Q\right)= 7\left(\begin{array}{rrrr}
	-9&-2&-6&1\\
	-2&-9&-6&1\\
	-6&-6&-11&3\\
	1&1&3&-7
\end{array}\right) \]
One checks that $\sigma(\mathcal{I})=-4$, so $W$ is a negative definite cobordism with boundary $Y\sqcup -\Sigma$. This will enable us to relate the $d$-invariants of $\Sigma$ with those of $Y$, once we specify a $\spinc$-structure on $W$.

Let $w\in H_2(W,\partial W)$ denote the class represented by the core of the $2$-handle added along $e_4=x_2$ (extended so that its boundary lies in $\partial W$). Hence $\partial_*(w)=(x_2,0)$ in $H_1(\partial W)\cong H_1(\Sigma)\oplus H_1(Y)$. We claim that the Poincare dual, $\hat{w}$, of $w$ is a characteristic class in $H^2(W)$ (as defined in Equation~\ref{eq:char} ).
It suffices to see that $7\hat{w}$ is characteristic. Note that $7w=\pi_*(E_4)$ where $\pi_*:H_2(W)\to H_2(W,\partial W)$. We must verify that, for all $x\in H_2(W)$,
$$
x\cdot x\equiv (\kappa\circ 7\hat{w})(x)= (\kappa\circ PD(7w))(x) = (\kappa\circ PD\circ \pi_*(E_4))(x)= E_4\cdot x,
$$
where the last equality is the definition of the intersection form (see ~\ref{eq:defintersform}). Suppose $x$ has coordinates $(a,b,c,d)$ with respect to our basis $\{E_1,E_2,E_3,E_4\}$. Then we calculate using $\mathcal{I}$ that, modulo $2$, both $x\cdot x$ and $E_4\cdot x$ are given by $a+b+c+d$. Thus $w$ is characteristic and so we know by Diagram~\ref{diag:spincstuff} and the surrounding discussion that there exists $\t\in \spinc(W)$ with $c_1(\t)=\hat{w}$. Then $\t$ restricts to some $(\s_0+z',\s_0^Y+z)\in\spinc(\Sigma\sqcup Y)$  where $s_0$ and $s_0^Y$ are $\spinc$-structures on $\Sigma$ and $Y$ that correspond to spin structures, and where $c_1(\s_0+z')=\hat{x_2}$ and $c_1(\s_0^Y+z)=0$. Since $c_1(\s_0+4\hat{x}_2)=2(4\hat{x}_2)=\hat{x}_2$ and since $H^2(\Sigma)$ has no $2$-torsion, we conclude that $z'=4\hat{x}_2$. Since $c_1(z_0^Y+z)=2z$ we conclude that either $z=0$ or $z=\hat{c}$ where $c$ is the unique element of order two in $H_1(Y)\cong \Z_7\oplus \Z_6$.

Moreover, since $\pi_*(E_4)=7w$, $\pi^*(\hat{E}_4)=7\hat{w}$. Thus, by definition (see ~\ref{eq:squareddef} and ~\ref{eq:QX}),
$$
c_1(\t)^2=\hat{w}^2:=\frac{Q_W(\hat{E}_4,\hat{E}_4)}{49}:=\frac{E_4\cdot E_4}{49}=-1.
$$
Thus, by the Ozsv\'ath-Szab\'o inequality ~\ref{eq:OSinequality}, 

\begin{equation}\label{eq:dsigmady}
	d(\Sigma,\s_0 + 4\hat{x}_2)+\frac{-1+4}{4}\leq d(Y,\s^Y_0+z)
\end{equation}

\begin{lem}\label{lem:positiveCrossingDInvariantLemma}
	Let $M$ be the $3$-manifold obtained by substituting two unknots for the two right-most copies of $T$ in Figure~\ref{topboundary}. Suppose $\s_0^M\in\spinc(M)$ comes from a spin structure on $M$. Then $d(Y,\s_0^Y+z) \leq d(M,\s^M_0+z')$, where, if $z=0$ then $z'=0$ and if $z=\hat{c}$ then $z'=\hat{c}'$ where $c'$ is the element of order two in $H_1(M)$.
\end{lem}
\begin{proof} Recall that the knot $T$ can be unknotted by changing one positive crossing.  We will describe a $4$-dimensional cobordism $Z$ from $Y$ to $M$.
	
	First, note that blowing down a $-1$-framed circle around a positive crossing in a framed knot (similar to Figure~\ref{blowdown}) changes it to a negative crossing and does not change the framing (since we choose the circle to have zero linking number with the knot).  Second, consider this process applied to the $(2,0)$-cable of such a knot (where the two components can have different framings).  One can check that blowing down a $-1$ around a positive crossing  in the ``doubled version'' of Figure~\ref{blowdown} does not change either the linking number between the two components or the framings of the components.  It does however change the crossing.  Thus, since $T$ can be unknotted by changing one positive crossing, the framed knot $(T,1)$ can be changed to $(U,1)$ by blowing down a single $-1$.  Similarly, the framed link $(T_{(2,0)},(2,1))$ can be changed to $(U_{(2,0)},(2,1))$ by blowing down one $-1$.  
	
Consider the cobordism $Z$ obtained by adding two $-1$-framed 2-handles to $Y\times[0,1]$ according to the previous paragraph. Since blowing down the $-1$'s amounts to unknotting the two right-most copies of $T$ in Figure~\ref{topboundary}, the top boundary of $Z$ is what we called $M$. Each of the $-1$ framed circles is nullhomologous in $Y$ and they have zero linking number in $Y$, so the intersection form on $H_2(Z)\cong \Z^2$ is given by the matrix $-I_{2x2}$.  It also follows that $H_1(Y)\cong H_1(Z)\cong H_1(M)$.
	
The $\spinc$-structure $(\mathfrak{s}_0^M,\s_0^Y)$ on $\del Z=M\sqcup -Y$ extends to some $\spinc$-structure $\mathfrak{t}_0$ on $Z$.  First we consider the case that $z=0$. By Lemma~\ref{lem:smallestchar}, we may alter $\t_0$ to assume $c_1(\mathfrak{t}_0)^2 = -\beta_2(Z)$.  Then by inequality~\ref{eq:OSinequality}, $d(Y,\s_0^Y) \leq d(M,\s_0^M)$. Next we consider the case that $z=\hat{c}$. Let $c'$ denote the unique element of order $2$ in $H_1(M)$  that is homologous to $c$ in $H_1(Z)$.  Since $(c', -c)\in H_1(\partial Z)$ maps to zero in $H_1(Z)$, the class $(\hat{c}', -\hat{c})\in H^2(Z,\partial Z)$ is $j^*(\beta)$ for some $\beta\in H^2(Z)$. Therefore $\t_0+\beta\in \spinc(Z)$ restricts to the class $(\mathfrak{s}_0^M+\hat{c}',\s_0^Y+\hat{c})$ in $\spinc(\partial Z)$. By Lemma~\ref{lem:smallestchar}, we may assume $c_1(\mathfrak{t}_0+\beta)^2 = -\beta_2(Z)$.  By inequality~\ref{eq:OSinequality}, $d(Y,\s_0^Y+\hat{c}) \leq d(M,\s_0^M+\hat{c}')$.
\end{proof}

Now one can simplify $M$ by blowing down the two $+1$'s. Then we see that 
$$
M\cong S^3_{-7}(T\#LHT)\#S^3_{-6}(LHT).
$$

Let $N=S^3_{-7}(T\#LHT)$.  To a knot $J$ in $S^3$, Ozsv\'ath and Szab\'o associate a $\Z\oplus\Z$-filtered chain complex $\cfkinf(J)$ over $\Z_2$.  There is an action on $\cfkinf(J)$ by a formal variable $U$ that lowers the bifiltration by $(1,1)$.  We can think of the generators of $\cfkinf(J)$ as triples $[\x,i,j]$ where $(i,j)$ is the bifiltration.  Let $\cfkinf(J)\{i\geq 0\ \mathrm{and}\ j\geq 0\}$ denote the quotient complex of $\cfkinf(J)$ generated by those generators of $\cfkinf(J)$ with $i\geq 0$ and $j\geq 0$.  Ozsv\'ath and Szab\'o give~\cite[Corollary 4.2]{OzSz:KnotInv} a filtered chain homotopy equivalence (denoted $\simeq$) 
	\begin{equation}\label{eq:ChainHomotopyEquivalence}
		HF_\ell^+\left(S^3_{-7}(J),\sss_0\right)\simeq H_k\left(\cfkinf(J)\{i\geq 0\ \mathrm{and}\ j\geq 0\}\right)
	\end{equation}
	where $\ell=k-3/2$.
Recall that the variable $U$ acts on $HF^+$ and that any element $\alpha\in HF^+(N,\sss_0)$ has a grading $gr(\alpha)\in\Q$.  The homotopy equivalence above respects the $U$-action.  The definition of the $d$-invariant is $$d(N,\sss_0)=\min_{\alpha\neq0\in HF^+(N,\sss_0)}\left\{gr(\alpha):\ \alpha\in\im U^k\ \mathrm{for\ all\ } k\geq 0\right\}$$  Thus, to compute $d(N,\sss_0)$, it suffices to consider the complex $\cfkinf(T\#LHT)$ and its quotients.
	
By Subsection~\ref{splitting} below, $\cfkinf(T)=\cfkinf(RHT)\oplus A$ where $A$ is acyclic.  Also, there is a K\"unneth theorem for $\cfkinf$ of connected sums.	These remarks imply:
	\begin{eqnarray*}
		\cfkinf(T\#LHT) &\simeq& \cfkinf(T) \otimes \cfkinf(LHT)\\
		& = & \left(\cfkinf(RHT) \oplus A\right) \otimes \cfkinf(LHT)\\
		& = & \left(\cfkinf(RHT)\otimes\cfkinf(LHT)\right) \oplus \left(A \otimes \cfkinf(LHT)\right)
	\end{eqnarray*}
By the K\"unneth theorem for complexes over the PID $\Z/2\Z\left[U,U\inv\right]$, the right summand is acyclic.  Since the $d$-invariant is defined as a minimum over nonzero homology classes, only $\cfkinf(RHT)\otimes\cfkinf(LHT)$ will contribute to the $d$-invariant.  Thus, as far as $d$-invariants are concerned, $\cfkinf(T\#LHT)$ and $\cfkinf(RHT\#LHT)$ are interchangeable.  This and the Ozsv\'ath-Szab\'o equivalence~(\ref{eq:ChainHomotopyEquivalence}) imply that $d(S^3_{-7}(T\#LHT), \sss_0) = d(S^3_{-7}(RHT\#LHT), \sss_0)$.  Since the $d$-invariant of surgery on a knot is an invariant of the smooth concordance class of the knot, we see that $d\left(S^3_{-7}(RHT\#LHT),\sss_0\right)=d\left(S^3_{-7}(U),\sss_0\right)$.  We see $d\left(S^3_{-7}(U),\sss_0\right)=-3/2$ by~\cite[Theorem 6.1]{OwSt:Lattice} (which is based on results in~\cite{OzSz:LensSpSur} and~\cite{OzSz:IntegerSurgeries}).

Finally, by combining the above calculations we can finish the proof of Theorem~\ref{thm:dinvtcalc}. First we treat the case that $z=z'=0$.
	The $d$-invariants of surgery on $L$-space knots are known.  In particular, $d(S^3_6(RHT),\sss_0) = 5/4 - 2 $~\cite[Theorem 6.1]{OwSt:Lattice}.  Thus by equation~\ref{eq:dsigmady}, Lemma~\ref{lem:positiveCrossingDInvariantLemma}, we have
	\begin{eqnarray*}
		d(\Sigma, \sss_0+4\widehat{x_2}) + 3/4 &\leq& d(Y,\sss_0) \\
		&\leq& d(M,\sss_0) \\
		&=& d(S^3_{-7}(T\#LHT),\sss_0)+d(S^3_{-6}(LHT),\sss_0) \\
		&=& d(S^3_{-7}(U),\sss_0)-d(S^3_{6}(RHT),\sss_0) \\
		&=& -3/2 - (5/4-2) \\
		&=& -3/4
	\end{eqnarray*}
Now we consider the case that $z=c,z'=c'$. In this case the only change is that we must now use $d(S^3_6(RHT),\sss_0+3)$ which was computed to be $-1/4$ in ~\cite[Theorem 6.1]{OwSt:Lattice}. Then as above
\begin{eqnarray*}
		d(\Sigma, \sss_0+4\widehat{x_2}) + 3/4 &\leq& d(Y,\sss_0+z) \\
		&\leq& d(M,\sss_0+z') \\
		&=& d(S^3_{-7}(T\#LHT),\sss_0)+d(S^3_{-6}(LHT),\sss_0+3) \\
		&=& d(S^3_{-7}(U),\sss_0)-d(S^3_{6}(RHT),\sss_0+3) \\
		&=& -3/2 - (-1/4)\\
		&=& -5/4
	\end{eqnarray*}

\end{proof}

\subsection{A splitting of $\cfkinf(T)$}\label{splitting}

We will briefly review the basics of $\cfkinf$ and prove a splitting result for $\cfkinf(T)$.  Throughout, abbreviate $C=\cfkinf$.  Our treatment here is a ``user's guide'' of sorts.  We refer the reader to the source \cite{OzSz:KnotInv} for more.  We would like to thank Matt Hedden and Jen Hom for telling us about this result (see ~\cite[Prop. 6.1]{HeKiLiv}).

Given a knot $K$, $C(K)$ is a $\Z\oplus\Z$-filtered chain complex over $\Z/2\Z$.  Its total homology is $HFK^\infty(S^3)$.  A formal variable $U$ acts on $C(K)$, and so does $U\inv$.  The $U$-action commutes with $\del$, making $C(K)$ a $\Z/2\Z\left[U,U\inv\right]$-module whose homology has rank $1$.  There is an ordering $\leq$ on the bifiltration: filtration level $(i,j)$ is $\leq$ filtration level $(i' j')$ if $i\leq i'$ and $j\leq j'$.  The boundary operator preserves $\leq$ on the level of filtrations, and $U$ lowers the bifiltration by $(1,1)$.  In fact, one can obtain a basis (as a module) for $C(K)$ from the $\Z$-filtered homology group $\widehat{HFK}(K)$ via the identification \cite[Equation 6.3]{HLR} 
	\begin{equation}\label{eq:CFKINFbasis}
		\widehat{HFK}_\ast(K,j) \cong C_\ast\{0,j\}
	\end{equation}
Using this identification, the $i$-coordinate of the bifiltration records the negative $U$-power, and the $j$-coordinate records the difference between the Alexander filtration (from $\widehat{HFK}$) and the $U$-power.  Multiplication by $U$ lowers homological grading by $2$, and so we have
	\begin{equation}\label{eq:FullCFKINF}
		C_\ast\{0,j\} \cong C_{\ast -2n}\{-n, j-n\}
	\end{equation}

Recall $T$ is the positively-clasped, untwisted Whitehead double of the right-handed trefoil $RHT$.  As a $\Z/2\Z$-vector space, $\widehat{HFK}(T)$ has rank $15$ \cite[Theorem 1.2]{Hed:WHD}, and so by equation~\ref{eq:CFKINFbasis}, $C(T)$ has rank $15$ as a $\Z/2\Z\left[U,U\inv\right]$-module.  It is convenient to visualize $C(T)$ by drawing the $\Z/2\Z$-ranks of the chain groups at filtration level $(i,j)$ in the $(i,j)$-plane.  We have done this in Figure~\ref{fig:cfkinft}, only for homological gradings $0$ (in Hindu-Arabic numerals) and $1$ (in Roman numerals).

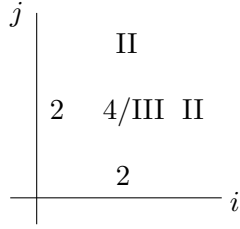
\begin{figure}[!ht]
	\begin{center}
		\begin{picture}(80,80)(0,0)
			\put(10,0){\line(0,1){80}}
			\put(0,10){\line(1,0){80}}
			\put(83,5){$i$}
			\put(0,77){$j$}
			\put(35,40){$4/\textrm{III}$}
			\put(15,40){$2$}
			\put(40,15){$2$}
			\put(40,65){$\textrm{II}$}
			\put(65,40){$\textrm{II}$}
		\end{picture}
		\caption{Ranks of the chain groups homological gradings $0$ and $\textrm{I}$}
		\label{fig:cfkinft}
	\end{center}
\end{figure}

We may draw arrows in the $(i,j)$ plane to indicate pieces the boundary map $\del$.  Since $\del$ preserves the filtration, arrows point down and/or left.  By the distribution of the chain groups in the $(i,j)$-plane, the only possible arrows are length one arrows $\downarrow$, $\leftarrow$, and $\swarrow$.  Let us denote the direction of these arrows by $\xrightarrow{d}$, $\xrightarrow{l}$, and $\xrightarrow{ld}$, respectively.  Note that by equation~\ref{eq:CFKINFbasis}, the homology of the column $C\{i=0\}$ is $\widehat{HF}(S^3)\cong \Z/2\Z$.  By the symmetry $\widehat{HFK}(K,-j)\cong\widehat{HFK}(K,j)$ and by the $U$-action, we see that the homology of the row $C\{j=0\}$ is $\widehat{HF}(S^3)$.  By the action, any column or row has homology $\Z/2\Z$.

Let us examine the column $C\{i=0\}$.  It is convenient to consider the ranks of the groups in this column, according to their $j$-filtration and homological grading $\ast$.  Below, $n_{j,\ast}$ denotes a rank $n$ $\Z/2\Z$-vector space in filtration level $j$ with homological grading $\ast$.  For typographical reasons, we will express the column in horizontal form so that the downward arrows in the column are expressed as rightward arrows.  Since the homology of the column has rank $1$, we see that the ranks of the vertical differential are as follows (two arrows means the differential has rank $2$):

\begin{figure}[!ht]
	\begin{center}
		\begin{picture}(100,30)(0,0)
			\put(0,15){$2_{1,0}$}
			\put(24,18){\vector(1,0){15}}
			\put(0,0){$2_{1,-1}$}
			\put(24,3){\vector(1,0){15}}
			\put(24,0){\vector(1,0){15}}
			\put(45,15){$3_{0,-1}$}
			\put(70,18){\vector(1,0){15}}
			\put(70,15){\vector(1,0){15}}
			\put(70,3){\vector(1,0){15}}
			\put(70,0){\vector(1,0){15}}
			\put(45,0){$4_{0,-2}$}
			\put(90,15){$2_{-1,-2}$}
			\put(90,0){$2_{-1,-3}$}
		\end{picture}
	\end{center}
\end{figure}

So a nonzero element, say $a$ in $2_{1,0}$ is in the kernel, but not the image, of the vertical differential.  Let us now return to considering all of $C(T)$.

Claim 1: there is only one arrow into $a$.  More specifically, there exists a unique arrow into $a$, and it is of the form $b\xrightarrow{l}a$.  We know there must be at least one $\xrightarrow{l}$ into $a$, by the symmetry of the filtration, for $a$ lies in the left-most position in its row and the vertical chain complex maps onto the bottom-most groups.  Thus, we may pick a $b\xrightarrow{l}a$.  If there is another arrow $d\xrightarrow{l}a$, we may eliminate it by a basis change.  If there is a diagonal $d\xrightarrow{ld}a$, we may eliminate it by a filtered change of basis ($b$ is below $d$, so $b \leq d$ on the level of filtrations).

Claim 2: there is no arrow out of $a$.  To see this, one must translate the $15$ generators in Figure~\ref{fig:cfkinft} by the action of $U\inv$.  Nothing lies to the left of $a$.  By our choice of $a$, there are no vertical arrows emanating from $a$.  No diagonal arrows may emanate from $a$, either, as the only elements southwest of $a$ have the wrong homological grading.

Claim 3: there is no arrow into $b$, by claims 1 and 2, and $\del^2=0$.

Claim 4: there is a (unique) $b\xrightarrow{d}c$.  Since $b$ is not in the image of the vertical differential (claim 3), and since it is not in the top of its column, $[b]=0\in H_\ast(C\{i=1\})$.  Thus, there is a nonzero $c$ so that $b\xrightarrow{d}c$.  We may assume $c$ is the vertical image of $b$.

Claim 5: there are no arrows emanating from $c$, because $0=\del^2 b = \del(a+c) = \del a + \del c = \del c$.

Claim 6: there are no other vertical arrows into $c$.  Let us look at the piece of $C\{i=1\}$ that splits as $2\to 3\to 2$.  This is the piece where $U\inv a$, $b$, and $c$ live.  Now $\del U\inv a =0$, but the $2\to 3$ has rank $1$.  Complete $U\inv a$ to a basis $\{U\inv a, g\}$.  Define $d$ to be the image of $g$.  Now $d$ is in the kernel of the vertical differential.  The map $3\to 2$ has rank $2$, so we may complete $\{b,d\}$ to a basis $\{b,d,e\}$ where $e\xrightarrow{d}f$ and $\{c,f\}$ is a basis for the bottom-most group in the column.  By a change of basis, we may assume that $e \not\xrightarrow{d}c$.

Claim 7: there are no diagonal arrows into $c$.  If there is a $h\xrightarrow{ld}c$, then $h$ must live in the bottom-most spot in its column.  There could be at most two diagonal arrows into $c$, so by a change of basis, we may assume there is exactly one.  By rank reasons, there would be a $i\xrightarrow{d}h$.  Thus, there would exist a two-arrow path from $i$ to $c$.  Now $\del^2 i = 0$, so there must exist another two-arrow path from $i$ to $c$.  Such a path must be $\xrightarrow{ld}\xrightarrow{d}$ or $\xrightarrow{d}\xrightarrow{ld}$.  The first case is impossible since $d\not\xrightarrow{d}c$ and $e\not\xrightarrow{d}$, and $b$ has no arrows coming in (claim 3).  The second case is impossible since we assumed there was a unique diagonal arrow into $c$.

Now, claims 1-7 imply that $C(T)$ splits off a summand consisting of $U$-translates of $a\xleftarrow{l}b\xrightarrow{d}c$.  This is $\cfkinf(RHT)$, as one can compute from the genus one Heegaard diagram for the trefoil.  The total homology of $C(T)$ has rank $1$, and so does $\cfkinf(RHT)$.  Thus, $\cfkinf(T)=\cfkinf(RHT)\oplus A$ where $A$ is acyclic.

\section{Appendix II: Casson-Gordon invariants as obstructions to membership in $\cP_2$ or $\cN_2$}\label{sec:CGinvts}

We review Casson-Gordon invariants. Suppose $K$ is a knot and $q$ is a prime power. Let $\Sigma_q$ denote the $q$-fold cyclic  cover of $S^3$  branched over $K$. Suppose $d=p^r$ is a prime power and $\chi:H_1(\Sigma_q)\to \Z_d$ is a homomorphism. The \textbf{associated character} is the composition $\overline{\chi}:H_1(\Sigma_q)\to \Z_d\hookrightarrow \mathbb{C}^*$ where $1\in \Z_d$ is mapped to a fixed primitive $d^{th}$ root of unity $\zeta_d$. To this data Casson-Gordon ~\cite{CG2} associate a Witt class
$$
\tau(K, \overline{\chi})\in L_0(\mathbf{k})\otimes \Z\left[\frac{1}{d}\right ],
$$
where $\mathbf{k}=\Q(\zeta_d)(t)$. We remark that, since it can be shown that $L_0(\mathbf{k})$ has no odd torsion ~\cite[Appendix 2]{Lith1}\cite[p.478]{GL1},  one  thinks, without loss,  of $\tau$ as lying in $L_0(\mathbf{k})\otimes R$ where $R=\Z_{(2)}$ if $p$ is odd and $R=\Q$ if $p=2$. We call this a \textbf{Casson-Gordon invariant}. It is not a concordance invariant. The collection of all such invariants (fixed $q$ and $d$, varying $\chi$)  obstructs a knot's being a slice knot in the following sense. Suppose  that $K$ bounds the slice disk $\Delta\hookrightarrow B^4$ and suppose that $V_q$ is the  $q$-fold cyclic  cover of $B^4$  branched over $\Delta$. Casson-Gordon proved ~\cite[Theorem 2]{CG2}:

\begin{itemize}
\item [A.] The kernel $Q$ of the inclusion-induced $j_*: H_1(\Sigma_q)\to H_1(V_q)$ is a metabolizer for the torsion-linking form, and $\overline{\chi}$  extends  to $H_1(V_q)$ if and only if $\chi(Q)=0$.
\item [B.] for any $\overline{\chi}$ that extends to $H_1(V_q)$,  $\tau(K,\overline{ \chi})=0$.
\end{itemize}

We call this ``the Casson-Gordon \textbf{slicing obstructions} for $K$''. We need to generalize these results by replacing the condition that $K$ is a slice knot by the much weaker conditions that, for part $A$, $K\in \mathcal{F}^{odd}_{1}$ and, for part $B$,  $K\in \mathcal{F}^{odd}_{1.5}$. These were essentially already proved in Proposition 9.7 and Theorem 9.11 of ~\cite{COT}, respectively, where these were shown under the hypotheses that $K\in \mathcal{F}_{1}$, and $K\in \mathcal{F}_{1.5}$ respectively. 

\begin{thm}\label{thm:CGvanish}  Suppose that $K\in \mathcal{F}^{odd}_{1.5}$, that is, $K$ is slice in a $4$-manifold $V$ as in Definition~\ref{def:altdefoddsolv}). Then ``the Casson-Gordon obstructions for $K$ vanish''. More precisely, with the above terminology:
\begin{itemize}
\item [A.] The kernel $Q$ of the inclusion-induced $j_*: H_1(\Sigma_q)\to H_1(V_q)$ satisfies $Q=Q^\perp$ with  respect to the torsion-linking form, and $\overline{\chi}$  extends  to $H_1(V_q)$ if and only if $\chi(Q)=0$.
\item [B.] for any $\overline{\chi}$ that extends to $H_1(V_q)$,  $\tau(K,\overline{ \chi})=0$.
\end{itemize}
Moreover part $A$ requires only that $K\in \mathcal{F}^{odd}_{1}$, and all statements hold even in the topological category (that is $V$ may be merely a topological $4$-manifold).
\end{thm}

\begin{proof}  Since $K\in \mathcal{F}^{odd}_{1.5}$, $K$ bounds a slice disk $\Delta$ in a manifold $V$ as in Definition~\ref{def:altdefoddsolv}  (for n=1.5). Alternatively the $4$-manifold $W=V\setminus \Delta$ has boundary $M_K$ and satisfies the conditions of Definition~\ref{def:nsolvableodd}. Let $W_q$ denote the $q$-fold cyclic cover of $W$ and let $V_q$ denote the branched cyclic cover over $\Delta$. Since there is a canonical isomorphism $H_1(V_q)\oplus \Z\cong H_1(W_q)$, $\overline{\chi}$ extends to $W_q$.

To prove $A$ we need only assume $K\in \mathcal{F}^{odd}_{1}$, that is that $V$ satisfies Definition~\ref{def:altdefoddsolv}  for n=1. This proof was sketched in ~\cite[Proposition 9.7]{COT}.  By Lemma~\ref{lem:onepositivecohomology}, it suffices to show that the intersection form on $V_q$ is unimodular.  The hypotheses on the $L_i$ and $D_i$, guaranteed by Definition~\ref{def:altdefoddsolv}, ensure that these surfaces  lift to any abelian covering space of $W$ and in particular to $W_q$. The $mq$ lifts of the $L_i$ and the $mq$ lifts of the $D_i$ span a free abelian subgroup of rank $2mq$ in $H_2(V_q)$ (they are linearly independent because each member has a homological dual).  Their span can be shown to be a direct summand of $H_2(V_q)$ by an argument similar to that at the start of the proof of Theorem~\ref{thm:onepositivedinvariants}.  An Euler characteristic argument shows that the dimension of $H_2(W_q;\Q)$ (and hence $H_2(V_q)$) is $2mq$. Thus these lifts form a basis for $H_2(V_q)/TH_2(V_q)$. Therefore the intersection form on $V_q$ is a direct sum of matrices of the form
$$
\left(\begin{array}{cc}
0 & 1\\
1 & *
\end{array}\right),
$$
hence is unimodular. This concludes the proof of $A$.

Our proof of $B$  will consist in pointing out why the proof of  Theorem 9.11 of ~\cite{COT} works under the weaker hypothesis. The reader can  follow along with ~\cite[p.516]{COT}, as we review the argument.   Recall that, since $\overline{\chi}$ extends to $H_1(W_q)$,  $\tau(K, \overline{\chi})$ may be calculated as the difference between  two Witt classes, one represented by the (image of the) ordinary intersection form on $H_2(W_q;\Q)$ and the other represented by the intersection form on $H_2(W_q;\mathbf{k})$. 

In order to show that the ordinary intersection form on $H_2(W_q;\Q)$ is Witt trivial it suffices to exhibit a half-rank isotropic subspace of $H_2(W_q;\Q)$. The hypotheses on the $L_i$ and $D_i$ ensure that these surfaces lift to any abelian covering space of $W$ and in particular to $W_q$. The $mq$ lifts of the $L_i$ form an isotropic subspace of dimension $mq$ (they are linearly independent because they have homological duals).  An Euler characteristic argument shows that the dimension of $H_2(W_q;\Q)$ is $2mq$. Thus the ordinary intersection form on $H_2(W_q;\Q)$ is Witt trivial.

Similarly the hypotheses imply that the $L_i$ lift to any \textit{metabelian} cover and indeed represent an isotropic submodule of the covering space of $W$ with $\pi_1=\pi_1(W)^{(2)}$. It follows that the $L_i$ span an isotropic subspace of $H_2(W_q;\mathbf{k})$. We claim that it is of half-rank.  By ~\cite[Lemma 9.6]{COT}, dim$_{\mathbf{k}}H_2(W_q,\mathbf{k})=$dim$_\Q H_2(W_q,\Q)=2mq$. Therefore it suffices to show that the $mq$ (disjoint) lifts of the $L_i$ are linearly independent in $H_2(W_q,\mathbf{k})$. This is proved on the bottom of page 516 in ~\cite{COT}  (in that argument replace the $2$-spheres by our surfaces $L_i$).
\end{proof}

Suppose that $K=R(J,\eta)$ is the result of infection on a ribbon knot $R$ by a knot $J$ along the circle $\eta$ where lk($\eta$,$R)=0$. In other words, $K$ is a satellite of $J$ with companion $R$ of winding number zero.
Suppose $q$ is a prime-power. Then the first homology of the $q$-fold branched covers of $S^3$ over $R$ and $R(K)$ are naturally isomorphic and so we can identify, in a canonical way, the characters on these $3$-manifolds ~\cite[Lemma 4]{Lith1}. Then, for any such character, $\chi$, of prime-power order $d$, it was shown in ~\cite[Corollary 2]{Lith1} that
$$
\tau(R(K),\chi)= \tau(R,\chi) + \sum_{i=1}^{q}\alpha_i(J)
$$
We must explain the notation $\alpha_i(J)$.  Each lift $\tilde{\eta}_i$ of $\eta$ represents a homology class $x_i\in H_1(\Sigma_q)$ and hence under $\chi$ corresponds to a root of unity $\chi(x_i)=\zeta_d^{s(i)}$. For each $i$ we have the following commutative diagram
$$
L_0(\Z[x,x^{-1}])\overset{f_i}{\longrightarrow} L_0(\Q(\zeta_d))\overset{i}{\longrightarrow} L_0(\Q(\zeta_d))\otimes \Z_{(2)}\overset{j}{\longrightarrow} L_0(\mathbf{k})\otimes \Z_{(2)},
$$
where the first map is induced by sending $x$ to $\chi(x_i)$ (her also for simplicity we restrict to the case that $d$ is odd). The class $\alpha_i(J)$ is the image under this composition of the ordinary algebraic concordance class of $J$. 

Suppose further that $\chi$ extends over the $q$-fold  cover of $B^4$ branched over  some slice disk for $R$. Then $\tau(R,\chi)=0$ by Casson-Gordon's main result.  Thus, in this case, the $\tau$-invariants of $R(K)$ are equal to sums of certain algebraic concordance invariants of $J$, where here one must use that the maps $i$ and $j$ above are injective ~\cite{Lith1}\cite[p480]{GL1}.  Specifically we are interested in certain $2$-torsion invariants, $\delta_{i/d}(J)$,  of the algebraic concordance class of $J$ which were defined by Gilmer-Livingston ~\cite[Section 1]{GL2}, and identified with $\Delta_J(\zeta_d^i)$ ~\cite[p.129]{GL2}\cite[Appendix 2]{GL1}. Here $\Delta_J(t)$ is the symmetrized Alexander polynomial of $J$. These invariants take values in the quotient, $N_d$, of $\Q(\zeta_d+\overline{\zeta}_d)$ by the subgroup of norms of elements of $\Q(\zeta_d)$.

Then, by combining our Theorem~\ref{thm:CGvanish} with the above-mentioned work of Litherland and Gilmer-Livingston, we have the following, which extends a version of ~\cite[Theorem 7]{GL2} (see their corrected statement in ~\cite{GL4}). 

\begin{thm}\label{thm:discrimvanish} Suppose that  $K=R(J,\eta)$ is the result of infection on a topologically slice knot $R$ by a knot $J$ along the circle $\eta$ where lk($\eta$,$R)=0$ and that $K\in \mathcal{F}^{odd}_{1.5}$ via the $4$-manifold $W$. Suppose $q$ is a prime power,  $\Sigma_q$ denotes the $q$-fold cyclic  cover  branched over $K$, $d$ is an odd prime power,  $\overline{\chi}:H_1(\Sigma_q)\to \Z_d\to \mathbb{C}^*$ is a  character, non-trivial on a lift of $\eta$, that extends over $W_q$ and extends over the $q$-fold  cover of $B^4$ branched over a slice disk for $R$.  Then
$$
\prod_{i=1}^{q}\Delta_J(\zeta_d^{s(i)})
$$
is a norm of $\Q(\zeta_d)$.
\end{thm}

\begin{cor}\label{cor:ourknotCG} Let $K$ be the knot shown on the left-hand side of Figure~\ref{fig:ourknotCG} where $J$ is the twist knot with a negative clasp and $11$ twists and $T$ is a topologically slice knot. Let $\eta_1, \eta_2$ be meridional circles to the left and right-hand bands of $K$, respectively. Then $K\notin \mathcal{N}_2$  via $K=\partial \Delta\hookrightarrow V$ wherein the kernel of the inclusion-induced map $j_*:H_1(\Sigma_3(K))\to H_1(V_3)$ is generated by the lifts of $\eta_2$; and $\Sigma_3, V_3$ are the $3$-fold covers branched over $K$ and $\Delta$ respectively.
\end{cor} 

\begin{figure}[htbp]
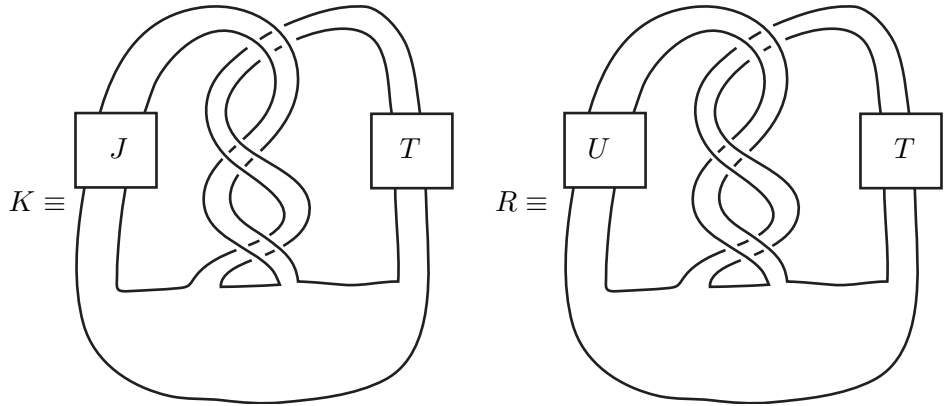

\setlength{\unitlength}{1pt}
\begin{picture}(327,151)
\put(0,0){\includegraphics{family_scaled}}
\put(185,0){\includegraphics{family_scaled}}
\put(13,93){$J$}
\put(-25,73){$K\equiv$}
\put(123,93){$T$}
\put(310,93){$T$}
\put(194,93){$U$}
\put(159,73){$R\equiv$}
\end{picture}
\caption{$K=R(J,\eta_1)$}\label{fig:ourknotCG}
\end{figure}

\begin{proof}[Proof of Corollary~\ref{cor:ourknotCG}]  Suppose that $K\in \mathcal{N}_2$  via such a $V$. Let $x_1,\tau(x_1)$ and $\tau^2(x_1)$ denote the homology classes of the lifts of $\eta_1$ and let $y_1,\tau(y_1)$ and $\tau^2(y_1)$ denote the homology classes of the lifts of $\eta_2$. Define $\chi:H_1(\Sigma_3(K))\to \Z_7$ by $\chi(x_1)=1, \chi(\tau(x_1))=4, \chi(\tau^2(x_1))=2, \chi(y_1)=0, \chi(\tau(y_1))=0,$ and  $\chi(\tau^2(y_1))=0$. Since ker$\chi\subset$ ker$j_*$, the character $\overline{\chi}$ extends to $V_3$.

Note that $K=R(J,\eta_1)$ where $R$ is the ribbon knot shown on the right-hand side of Figure~\ref{fig:ourknotCG}. Note that $R=S(T,\eta_2)$ where $S$ is the ribbon knot $9_{46}$. The latter knot has a ribbon disk obtained by ``cutting the right-hand band''. It may be seen then that $R$ is topologically slice via a disk $\Delta\hookrightarrow B^4$  for which $\chi$ extends over the $3$-fold cover branched over $\Delta$.
Now we may apply Theorem~\ref{thm:CGvanish} with $q=3, d=7$ to conclude that
$$
\Delta_J(\zeta_7)\Delta_J(\zeta_7^4)\Delta_J(\zeta_7^2)=11,089=13\cdot 853
$$
is a norm of $\Q(\zeta_7)$. Here $\Delta_J(t)=-11t+23-11t^{-1}$ and the calculation was done via Maple. But $13$ has order $2$ in $\Z_7^*$ so $11089$ is not a norm of $\Q(\zeta_7)$ ~\cite[p.128]{GL2}. This is a contradiction.
\end{proof}

\bibliographystyle{plain}
\bibliography{posfiltrbibJan.bib}
\end{document}